\documentclass[a4paper, 12pt, oneside, reqno]{amsart}

\usepackage[T1]{fontenc}
\usepackage[utf8]{inputenc}
\usepackage{mathptmx}       
\usepackage{helvet}         
\usepackage{courier}        
\usepackage{type1cm}        
\usepackage{graphicx}        
\usepackage{amssymb, amsmath}
\usepackage{enumitem}
\usepackage[DIV=10]{typearea}
\usepackage{mathtools}
\usepackage{hyperref}
\usepackage{tikz,tikz-cd}

\numberwithin{equation}{section}

\newcommand{\Z}{\mathbb{Z}}
\newcommand{\N}{\mathbb{N}}
\newcommand{\G}{\mathcal{G}}
\newcommand{\T}{\mathbb{T}}
\newcommand{\K}{\mathcal{K}}
\newcommand{\R}{\mathcal{R}}
\newcommand{\Oo}{\mathcal{O}}

\newcommand{\D}{\mathcal{D}}
\newcommand{\id}{{\operatorname{Id}}}
\newcommand{\KP}{\operatorname{KP}}
\newcommand{\X}{\overline{X}}
\newcommand{\tsh}{\overline{\sigma}}

\newcommand{\orb}{\operatorname{orb}}
\newcommand{\bps}{\operatorname{\partial \Lambda}}
\newcommand{\Lone}{\operatorname{\Lambda_1}}
\newcommand{\Ltwo}{\operatorname{\Lambda_2}}

\newcommand{\Per}{\operatorname{Per}}
\newcommand{\IP}{\operatorname{IP}}

\newcommand{\iso}{\operatorname{Iso}}
\newcommand{\clsp}{\overline{\operatorname{span}}}

\newtheorem{theorem}{Theorem}[section]
\newtheorem{proposition}[theorem]{Proposition}
\newtheorem{corollary}[theorem]{Corollary}
\newtheorem{lemma}[theorem]{Lemma}

\theoremstyle{definition}
\newtheorem{definition}[theorem]{Definition}
\newtheorem{remark}[theorem]{Remark}
\newtheorem{exm}[theorem]{Example}

\author{Toke Meier Carlsen}
\address{Department of Science and Technology\\University of the Faroe Islands\\
Vestara Bryggja 15\\ FO-100 T\'orshavn\\ Faroe Islands}
\email{toke.carlsen@gmail.com}

\author{James Rout}
\address{School of Mathematics and Statistics\\ University of New South Wales\\
Kensington \\ NSW 2052\\ Australia}
\email{jdr749@uowmail.edu.au}

\date{\today}
\subjclass[2010]{Primary: 46L55; Secondary: 37A55, 37B10, 16S99, 22A22}
\keywords{Higher-rank graph; orbit equivalence; conjugacy; multi-dimensional shift space; groupoid; $k$-graph $C^*$-algebra; Kumjian--Pask algebra}

\title{Orbit equivalence of higher-rank graphs}

\begin{document}

\begin{abstract}
We study the notions of continuous orbit equivalence and eventual one-sided conjugacy of finitely-aligned higher-rank graphs and two-sided conjugacy of row-finite higher-rank graphs with finitely many vertices and no sinks or sources. We show that there is a continuous orbit equivalence between two finitely-aligned higher-rank graphs that preserves the periodicity of boundary paths if and only if the boundary path groupoids are isomorphic, and characterise continuous orbit equivalence, eventual one-sided conjugacy, and two-sided conjugacy of higher-rank graphs in terms of the $C^*$-algebras and the Kumjian--Pask algebras of the higher-rank graphs.
\end{abstract}

\maketitle

\section{Introduction}

Matsumoto introduced and studied continuous orbit equivalence for irreducible one-sided shifts of finite type in \cite{Mat10}. He showed that one-sided shifts $X_A$ and $X_B$ are continuously orbit equivalent for finite irreducible $\{0,1\}$-matrices $A$ and $B$ that are not permutations if and only if there is an isomorphism between the corresponding Cuntz--Krieger algebras $\Oo_A$ and $\Oo_B$ that preserves the diagonal subalgebra. Brownlowe, Carlsen and Whittaker extended continuous orbit equivalence to boundary path spaces of arbitrary directed graphs in \cite{BCW}. They used groupoid techniques to show that if arbitrary directed graphs $E$ and $F$ are continuously orbit equivalent, then there is a diagonal-preserving isomorphism between the corresponding graph $C^*$-algebras $C^*(E)$ and $C^*(F)$. Moreover, they showed that the converse holds provided that the graphs both satisfy condition (L). 

Arklint, Eilers and Ruiz in \cite{AER} and Carlsen and Winger in \cite{CW}, independently and using different methods, showed that arbitrary graphs $E$ and $F$ are continuously orbit equivalent by an equivalence that maps eventually periodic points to eventually periodic points if and only if there is a diagonal-preserving isomorphism between $C^*(E)$ and $C^*(F)$. Continuous orbit equivalence has also been characterised in terms of diagonal-preserving isomorphisms of Leavitt path algebras. This follows from the work of Brown, Clark and an Huef for row-finite graphs \cite{BCH} and Carlsen and Rout for arbitrary graphs \cite{CR2}.

Matsumoto proved in \cite{Mat17} that if $A$ and $B$ are finite irreducible $\{0,1\}$-matrices that are not permutations matrices, then the one-sided shifts of finite type $X_A$ and $X_B$ are eventually conjugate if and only if there is an isomorphism between the Cuntz--Krieger algebras $\Oo_A$ and $\Oo_B$ that preserves the diagonal subalgebra and intertwines the gauge actions of $\Oo_A$ and $\Oo_B$. This result was generalised by Carlsen and Rout to arbitrary graphs in \cite{CR1}, and an analogous result for Leavitt path algebras was proved in \cite{CR2}.

In \cite{CR1} and \cite{CR2}, diagonal-preserving stable isomorphisms of graph graph $C^*$-algebras and Leavitt path algebras was characterised in terms of stable isomorphisms of the corresponding boundary path groupoids and it was shown that the two-sided shifts of finite type of two finite graphs with no sinks and no sources are conjugate if and only if there is a diagonal-preserving gauge-invariant isomorphism between the stabilised graph $C^*$-algebras, and if and only if there is a diagonal-preserving graded isomorphism between the stabilised Leavitt path algebras.

Higher-rank graphs or $k$-graphs are combinatorial objects that generalise directed graphs. These were introduced by Kumjian and Pask in \cite{KP} in order to model the higher-rank Cuntz--Krieger algebras of Robertson and Steger \cite{RS}. Kumjian and Pask originally associated a $C^*$-algebra $C^*(\Lambda)$ to a row-finite $k$-graph $\Lambda$ without sources. Raeburn, Sims and Yeend then associated $C^*$-algebras to larger classes of $k$-graphs in \cite{RSY1} and \cite{RSY}. The largest class is that of finitely-aligned $k$-graphs. Farthing, Muhly and Yeend showed that the $C^*$-algebra of a finitely-aligned $k$-graph can also be realised as the groupoid $C^*$-algebra of the boundary path groupoid $\G_\Lambda$ in \cite{FMY}. The algebraic analogues of $k$-graph $C^*$-algebras are the Kumjian--Pask algebras $\KP(\Lambda)$. These were introduced for row-finite $k$-graphs $\Lambda$ without sources by Aranda Pino, Clark, an Huef and Raeburn in \cite{APBaHR}. Clark and Pangalela recently associated Kumjian--Pask algebras to finitely-aligned $k$-graphs in \cite{CP}. They showed that $\KP(\Lambda)$ can also be realised as the Steinberg algebra of the boundary path groupoid $\G_\Lambda$.

\subsection*{The results of this paper}
In this paper, we use groupoid methods to study continuous orbit equivalence of finitely-aligned $k$-graphs (of possibly different rank) as well as eventually one-sided conjugacy of finitely-aligned $k$-graphs (of the same rank) and two-sided conjugacy of row-finite $k$-graphs (of the same rank) with finitely many vertices and no sinks or sources. 

In Section~\ref{sec:orbit}, we characterise continuous orbit equivalence between a finitely-aligned $k_1$-graph $\Lone$ and a finitely-aligned $k_2$-graph $\Ltwo$ in terms of groupoid isomorphisms of the boundary path groupoids $\G_{\Lone}$ and $\G_{\Ltwo}$. To do this, we recall the periodicity group of a boundary path and introduce a period-preserving condition for continuous orbit equivalence. We show that $\Lone$ and $\Ltwo$ are continuously orbit equivalent by an equivalence that is period-preserving if and only if the associated boundary path groupoids $\G_{\Lone}$ and $\G_{\Ltwo}$ are isomorphic (Proposition~\ref{prop:h}). 

We combine this with results about groupoid $C^*$-algebras in \cite{CRST} and Steinberg algebras in \cite{CR2} and \cite{S2} to see that $\Lone$ and $\Ltwo$ are continuously orbit equivalence by an equivalence that is period-preserving if and only if $C^*(\Lone)$ and $C^*(\Ltwo)$ are isomorphic by a diagonal-preserving isomorphism if and only if the Kumjian--Pask algebras $\KP(\Lone)$ and $\KP(\Ltwo)$ are isomorphic by a diagonal-preserving ring-isomorphism (Corollary~\ref{conj:2}). 

We give examples of $k$-graphs of different rank that are continuously orbit equivalent (Example~\ref{exm:coe}), and we give an example of $2$-graphs whose boundary path spaces are homeomorphic but not continuously orbit equivalent (Remark~\ref{rmk:nonexample}).

In Theorem~\ref{thm:eventual}, we show that finitely-aligned $k$-graphs $\Lone$ and $\Ltwo$ are eventually one-sided conjugate if and only if $C^*(\Lone)$ and $C^*(\Ltwo)$ are isomorphic by a diagonal-preserving isomorphism that intertwines the gauge actions if and only if $\KP(\Lone)$ and $\KP(\Ltwo)$ are isomorphic by a diagonal-preserving ring-isomorphism that preserves the $\Z^k$-grading.

We then use the results of \cite{CRS} to characterise diagonal-preserving stable isomorphisms of $k$-graph $C^*$-algebras and Kumjian--Pask algebras in terms of stable isomorphisms of the corresponding boundary path groupoids (Corollary~\ref{cor:kakutani}).

Finally, we show in Theorem~\ref{thm:conjugacy} that the two-sided multi-dimensional shift spaces associated to row-finite $k$-graphs $\Lone$ and $\Ltwo$ with finitely many vertices and no sinks or sources are conjugate if and only if $C^*(\Lone)$ and $C^*(\Ltwo)$ are stably isomorphic by a diagonal-preserving isomorphism that intertwines the gauge actions if and only if $\KP(\Lone)$ and $\KP(\Ltwo)$ are stably isomorphic by a diagonal-preserving isomorphism that preserves the $\Z^k$-grading.

\section{Preliminaries}

For the benefit of the reader, we recall the definitions of higher-rank graphs, their $C^*$-algebras and Kumjian--Pask algebras, their boundary path spaces and boundary path groupoids.

Let $\N$ be the set of positive integers together with $0$, and let $k$ be a positive integer. We regard $\N^k$ as a semigroup with identity $0$. We write $n \in \N^k$ as $n=(n_1, \dots, n_k)$. We use $\le$ for the partial order on $\N^k$ given by $m \le n$ if $m_i \le n_i$ for all $1 \le i \le k$. For $m,n \in \N^k$, we write $n \vee m$ for their coordinate-wise maximum and $n \wedge m$ for their coordinate-wise minimum.

\subsection{Higher-rank graphs and their $C^*$-algebras and Kumjian--Pask algebras}

A \emph{higher-rank graph} or \emph{$k$-graph} is a countable small category $\Lambda:=(\operatorname{obj}(\Lambda), \operatorname{mor}(\Lambda),r,s)$ together with a functor $d: \Lambda \to \mathbb{N}^k$ satisfying the \emph{factorisation property}: for every $\lambda \in \Lambda$ and $m,n \in \mathbb{N}^k$ with $d(\lambda)=m+n$, there are unique elements $\mu,\nu \in \Lambda$ with $d(\mu) = m$ and $d(\nu) = n$ such that $\lambda = \mu \nu$. We then write $\lambda(0,m) := \mu$ and $\lambda(m,m+n)=\nu$.

\begin{exm}[{\cite[Example~2.2.(ii)]{RSY1}}]
For $k \in \N$ and $m \in (\N \cup \{\infty\})^k$, define $\Omega_{k,m} := \{(p,q) \in \N^k \times \N^k: p \le q \le m \}.$ This is a $k$-graph with $\Omega_{k,m}^0=\{p \in \N^k: p \le m\}$, $r(p,q)=p$, $s(p,q)=q$, $d(p,q)=q-p$ and composition defined by $(p,q)(q,r)=(p,r)$. We write $\Omega_{k,\infty}$ for the $k$-graph $\Omega_{k,(\infty, \dots, \infty)}$ introduced in \cite[Example~1.7~(ii)]{KP}.
\end{exm}

For $n \in \N^k$, we define $\Lambda^n := \{\lambda \in \Lambda: d(\lambda)=n\}$. We regard the elements of $\Lambda^n$ as \emph{paths of length} $n$, and the elements of $\Lambda^0$ as \emph{vertices}. The uniqueness of the factorisation property allows us to identify $\operatorname{obj}(\Lambda)$ with $\Lambda^0$. For $\lambda \in \Lambda$, we refer to $r(\lambda)$ as the \emph{range} of $\lambda$ and $s(\lambda)$ as the \emph{source} of $\lambda$. For $v \in \Lambda^0$ and $E \subseteq \Lambda$, we define $v E := \{\lambda \in E: r(\lambda)=v \}$ and $E v := \{\lambda \in E: s(\lambda)=v\}$, and we define $\lambda E := \{ \lambda \mu: \mu \in s(\lambda) E \}$ and $E \lambda := \{\mu \lambda: \mu \in E r(\lambda)\}$. A vertex $v \in \Lambda^0$ is called a \emph{source} if $v \Lambda^m = \emptyset$ for some $m \in \N^k$, and called a \emph{sink} if $\Lambda^m v = \emptyset$ for some $m \in \N^k$. A $k$-graph $\Lambda$ is called \emph{row-finite} if for each $m\in \N^k$ and $v \in \Lambda^0$ the set $v \Lambda^m$ is finite.

For $\lambda, \mu \in \Lambda$, we say that $\tau \in \Lambda$ is a \emph{minimal common extension} of $\lambda$ and $\mu$ if $d(\tau) = d(\lambda) \vee d(\mu)$ and $\tau(0,d(\lambda)) = \lambda$ and $\tau(0,d(\mu))=\mu$. We write $\operatorname{MCE}(\lambda,\mu)$ for the collection of all minimal common extensions of $\lambda$ and $\mu$. We then define $\Lambda^{\operatorname{min}}(\lambda,\mu) := \{(\rho,\tau) \in \Lambda \times \Lambda: \lambda \rho = \mu \tau \in \operatorname{MCE}(\lambda,\mu) \}$. A $k$-graph $\Lambda$ is called \emph{finitely-aligned} if for all $\lambda, \mu \in \Lambda$, the set $\Lambda^{\operatorname{min}}(\lambda,\mu)$ is finite (possibly empty). A set $E \subseteq v \Lambda$ is said to be \emph{exhaustive} if for every $\lambda \in v \Lambda$, there exists $\mu \in E$ such that $\Lambda^{\operatorname{min}}(\lambda,\mu) \not = \emptyset$. 

We recall the $C^*$-algebra associated to a finitely-aligned $k$-graph introduced by Raeburn, Sims and Yeend in \cite{RSY}.

\begin{definition}[{{\cite[Definition~2.5]{RSY}}}]
Let $\Lambda$ be a finitely-aligned $k$-graph. A \emph{Cuntz--Krieger} $\Lambda$-family is a collection $\{S_\lambda : \lambda \in \Lambda \}$ of partial isometries in a $C^*$-algebra $A$ satisfying
\begin{enumerate}
\item $\{S_v: v \in \Lambda^0 \}$ is a collection of mutually orthogonal projections;
\item $S_\lambda S_\mu = S_{\lambda \mu}$ whenever $s(\lambda)=r(\mu)$;
\item $S_\lambda^* S_\mu = \sum_{(\alpha,\beta) \in \Lambda^{\operatorname{min}}(\lambda,\mu)} S_\alpha S_\beta^*$ for all $\lambda, \mu \in \Lambda$; and
\item $\prod_{\lambda \in E} (S_v - S_\lambda S_\lambda^*) =0$ for all $v \in \Lambda^0$ and finite exhaustive $E \subset v\Lambda$.
\end{enumerate}
\end{definition}

Given a finitely-aligned $k$-graph $\Lambda$, there is a $C^*$-algebra $C^*(\Lambda)$ generated by a Cuntz--Krieger $\Lambda$-family $\{s_\lambda: \lambda \in \Lambda \}$ which is universal in the following sense: for any Cuntz--Kriger $\Lambda$-family $\{S_\lambda: \lambda \in \Lambda \}$ in a $C^*$-algebra $A$, there is a unique homomorphism $\pi_S: C^*(\Lambda) \to A$ such that $\pi_S(s_\lambda)=S_\lambda$ for all $\lambda \in \Lambda$. By \cite[Lemma~2.7]{RSY}, $C^*(\Lambda) = \clsp \{s_\lambda s_\mu^*: \lambda, \mu \in \Lambda, s(\lambda)=s(\mu) \}$. The \emph{diagonal subalgebra} of $C^*(\Lambda)$ is given by $\D(\Lambda) := \clsp \{s_\mu s_\mu^*: \mu \in \Lambda \}$. In this paper, all isomorphisms of $k$-graph $C^*$-algebras are $*$-isomorphisms. For a finitely-aligned $k_1$-graph $\Lone$ and a finitely-aligned $k_2$-graph $\Ltwo$, we say that an isomorphism $\phi: C^*(\Lone) \to C^*(\Ltwo)$ is \emph{diagonal-preserving} if $\phi(\D(\Lone))=\D(\Ltwo)$.  

We recall the Kumjian--Pask algebra associated to a finitely-aligned $k$-graph $\Lambda$ introduced by Clark and Pangalela in \cite{CP}. For each $\lambda \in \Lambda \setminus \Lambda^0$, we define a \emph{ghost path} by a formal symbol $\lambda^*$. For $v \in \Lambda^0$, we define $v^* := v$, and extend $r$ and $s$ to the collection of ghost paths by setting $r(\lambda^*) := s(\lambda)$ and $s(\lambda) := r(\lambda^*)$. Composition of ghost paths is defined by $\lambda^* \mu^* := (\mu \lambda)^*$. The factorisation property of $\Lambda$ induces a similar factorisation property on the collection of ghost paths.

\begin{definition}[{{\cite[Definition~3.1]{CP}}}]
Let $\Lambda$ be a finitely-aligned $k$-graph, and $R$ a commutative ring with unit. A \emph{Kumjian--Pask} $\Lambda$-family $\{S_\lambda, S_{\mu^*}: \lambda, \mu \in \Lambda\}$ in an $R$-algebra $A$ consists of a function $S: \Lambda \cup \{\mu^*: \mu \in \Lambda \setminus \Lambda^0 \} \to A$ satisfying
\begin{enumerate}
\item $\{S_v: v \in \Lambda^0 \}$ is a collection of mutually orthogonal idempotents;
\item $S_\lambda S_\mu = S_{\lambda \mu}$ and $S_{\mu^*} S_{\lambda^*} = S_{(\lambda \mu)^*}$ whenever $s(\lambda)=r(\mu)$;
\item $S_{\lambda^*} S_\mu = \sum_{(\alpha,\beta) \in \Lambda^{\operatorname{min}}(\lambda,\mu)} S_\alpha S_{\beta^*}$ for all $\lambda, \mu \in \Lambda$; and
\item $\prod_{\lambda \in E} (S_v - S_\lambda S_{\lambda^*}) =0$ for all $v \in \Lambda^0$ and finite exhaustive $E \subset v\Lambda$.
\end{enumerate}
\end{definition}

Given a finitely-aligned $k$-graph and a commutative ring $R$ with unit, there is an $R$-algebra $\KP_R(\Lambda)$ generated by a Kumjian--Pask $\Lambda$-family $\{s_\lambda, s_{\mu^*}: \lambda,\mu \in \Lambda \}$ which is universal in the following sense: for any Kumjian--Pask $\Lambda$-family $\{S_\lambda, S_{\mu^*}: \lambda,\mu \in \Lambda \}$ in an $R$-algebra $A$, there is a unique $R$-algebra homomorphism $\pi_S: \KP_R(\Lambda) \to A$ such that $\pi_S(s_\lambda)=S_\lambda$ and $\pi_S(s_{\mu^*}) = S_{\mu^*}$ for all $\lambda,\mu \in \Lambda$ (see \cite[Theorem~3.7]{CP}). By \cite[Proposition~3.3]{CP}, $KP_R(\Lambda) = \operatorname{span}_R \{s_\lambda s_{\mu^*}: \lambda, \mu \in \Lambda, s(\lambda)=s(\mu) \}$. The \emph{diagonal subalgebra} of $\KP_R(\Lambda)$ is given by $D_R(\Lambda) := \operatorname{span}_R \{s_\mu s_{\mu^*}: \mu \in \Lambda \}$. In this paper, we will consider both ring-isomorphisms and $*$-algebra-isomorphisms of Kumjian--Pask algebras. For a finitely-aligned $k_1$-graph $\Lone$ and a finitely-aligned $k_2$-graph $\Ltwo$, we say that an isomorphism $\phi: \KP_R(\Lone) \to \KP_R(\Ltwo)$ is \emph{diagonal-preserving} if $\phi(D_R(\Lone))=D_R(\Ltwo)$.

\subsection{The boundary path groupoid of a higher-rank graph}

A \emph{boundary path} of a $k$-graph $\Lambda$ is a degree-preserving functor $x : \Omega_{k,m} \to \Lambda$ such that for every $n \in \mathbb{N}^k$ with $n \le m$ and every finite exhaustive subset $E \subset x(n, n) \Lambda$, there exists $\lambda \in E$ such that $x(n, n + d (\lambda)) = \lambda$ (see \cite[Definition~5.10]{FMY}). We write $\partial \Lambda$ for the set of all boundary paths. This is called the \emph{boundary path space} of $\Lambda$. We extend $r$ and $d$ to boundary paths $x: \Omega_{k,m} \to \Lambda$ by setting $r(x):=x(0,0)$ and $d(x) :=m$. Note that for $v \in \Lambda^0$, the set $v \partial \Lambda := \{ x \in \partial \Lambda: r(x)=v\}$ is nonempty (see \cite[Lemma~5.15]{FMY}). For $m\in\N^k$, we let $\partial\Lambda^{\ge m}:=\{x\in\partial\Lambda:d(x)\ge m\}$ and we let $\sigma_\Lambda^m:\partial\Lambda^{\ge m} \to \partial\Lambda$ be the \emph{shift map} defined by $\sigma_\Lambda^m(x)(p,q)=x(p+m,q+m)$ for $p,q\in\N^k$ with $p\le q\le d(x)-m$ (see \cite[Notation~5.1 and Lemma~5.13(1)]{Y}). We implicitly assume that $x\in\partial\Lambda^{\ge m}$ whenever we write $\sigma_\Lambda^m(x)$. For $\lambda \in \Lambda$ and $x \in s(\lambda) \partial \Lambda$, there exists $\lambda x \in \partial \Lambda$ such that $x=\sigma_\Lambda^{d(\lambda)}(\lambda x)$ and $\lambda = (\lambda x)(0,d(\lambda))$ (see \cite[Proposition~3.0.11]{WThesis}). Define $\lambda \partial \Lambda := \{ \lambda x: x \in s(\lambda) \partial \Lambda \}$. We also write $Z_\Lambda(\lambda) := \lambda \partial \Lambda$. For a finite nonexhaustive subset $G \subseteq s(\lambda) \Lambda$, we define \[Z_\Lambda(\lambda \setminus G) := Z_\Lambda(\lambda) \setminus \left(\bigcup_{\nu \in G} Z_\Lambda(\lambda \nu) \right).\] The sets $Z_\Lambda(\lambda \setminus G)$ form a basis for a locally compact Hausdorff topology on $\partial \Lambda$, and each $Z_\Lambda(\lambda \setminus G)$ is compact in this topology.

We recall the boundary path groupoid $\G_\Lambda$ of a $k$-graph $\Lambda$ as defined in \cite[Definition~4.8]{Y} (see also \cite[Example~5.2]{CP}).

The \emph{boundary path groupoid} $\G_\Lambda$ is given by
\begin{align*} 
\G_\Lambda := \{ (x,m,y) \in \partial \Lambda \times \Z^k \times \partial \Lambda: & \text{ there exist } p,q \in \N^k \text{ such that } \\ & p \le d(x), q \le d(y), p - q = m \text{ and } \sigma_\Lambda^p(x)=\sigma_\Lambda^q(y) \}, 
\end{align*} 
with partially-defined product $(x_1,m_1,y_1)(y_1, m_2,y_2) \mapsto (x_1, m_1+m_2, y_2)$, inverse operation $(x,m,y) \mapsto (y,-m,x)$, and range and source maps $r(x,m,y) := x$ and $s(x,m,y) := y$. We identify the boundary path space $\partial \Lambda$ with the unit space $\G_\Lambda^{0}$ via the map $x \to (x,0,x)$.

We describe a topology on $\G_\Lambda$ making it a locally compact, Hausdorff and ample groupoid (see \cite[Proposition~6.5 and Proposition~6.8]{FMY}). Write $\Lambda *_s \Lambda := \{ (\lambda, \mu) \in \Lambda \times \Lambda: s(\lambda)=s(\mu) \}$. For $(\lambda,\mu) \in \Lambda*_s \Lambda$ and a finite nonexhaustive subset $G \subseteq s(\lambda) \Lambda$, we write
\begin{align*} 
Z_\Lambda(\lambda *_s \mu) := \{(x,d(\lambda)-d(\mu),y) \in \G_\Lambda: \, & \, x \in Z_\Lambda(\lambda) \text{ and } y \in Z_\Lambda(\mu)\},
\end{align*} 
and
\[Z_\Lambda(\lambda*_s \mu \setminus G) := Z_\Lambda(\lambda *_s \mu) \setminus \left( \bigcup_{\nu \in G} Z_\Lambda(\lambda \nu *_s \mu \nu) \right). \]
The sets $Z_\Lambda(\lambda*_s \mu \setminus G)$ form a basis for a second-countable, Hausdorff topology on $\G_\Lambda$, and each $Z_\Lambda(\lambda *_s \mu \setminus G)$ is a compact open bisection. The relative topology on $\G_\Lambda^0$ agrees with the topology on $\partial \Lambda$.

The $C^*$-algebras of finitely-aligned $k$-graphs can be realised as the groupoid $C^*$-algebras of the boundary path groupoid, and the Kumjian--Pask algebras as the Steinberg algebras of the boundary path groupoid. More specifically, if $\Lambda$ is a finitely-aligned $k$-graph, then by \cite[Theorem~6.13]{FMY}, there is an isomorphism $\pi: C^*(\Lambda) \to C^*(\G_\Lambda)$ satisfying $\pi(s_\lambda) = 1_{Z_\Lambda(\lambda *_s s(\lambda))}$ for all $\lambda \in \Lambda$. Similarly, if $R$ is a commutative ring with unit, then by \cite[Proposition~5.4]{CP} there is an isomorphism $\pi_T: \KP_R(\Lambda) \to A_R (G_\Lambda)$ such that $\pi_T(s_\lambda) = 1_{Z_\Lambda(\lambda *_s s(\lambda))}$ and $\pi_T(s_{\lambda^*}) = 1_{Z_\Lambda(s(\lambda) *_s \lambda)}$ for $\lambda \in \Lambda$. The isomorphism $\pi$ carries $\D(\Lambda)$ to $C_0(\G_\Lambda^{0})$ viewed as a subalgebra of $C^*(\G_\Lambda)$, and likewise the isomorphism $\pi_T$ carries $D_R(\Lambda)$ to $A_R(\G_\Lambda^{0})$ viewed as a subalgebra of $A_R(\G_\Lambda)$. It follows from results about groupoid $C^*$-algebras in \cite{CRST} that the diagonal-preserving isomorphisms of $k$-graph $C^*$-algebras are characterised in terms of isomorphisms of the corresponding boundary path groupoids. Similarly, it follows from results about Steinberg algebras in \cite{CR2} that the ring-isomorphisms and $*$-algebra-isomorphisms of Kumjian--Pask algebras are also characterised in terms of isomorphisms of the corresponding boundary path groupoids.

We now describe another basis for the topology on $\G_\Lambda$ that will be useful in the proofs of Lemma~\ref{lem:l} and Theorem~\ref{thm:eventual}. For $m,n\in\mathbb{N}^k$, $U$ is an open subset of $\partial\Lambda^{\ge m}$ such that $\sigma_\Lambda^m$ is injective on $U$, and $V$ is an open subset of $\partial\Lambda^{\ge n}$ such $\sigma_\Lambda^n$ is injective on $V$ and $\sigma_\Lambda^m(U)=\sigma_\Lambda^n(V)$, we define $$Z(U,m,n,V):=\{(x,m-n,y):x\in U,\ y\in V,\ \sigma_\Lambda^m(x)=\sigma_\Lambda^n(y)\}.$$

\begin{lemma}\label{lem:groupoidbasis}
Let $\Lambda$ be a finitely-aligned $k$-graph. The sets $Z(U,m,n,V)$ form a basis for the topology of $\G_\Lambda$.
\end{lemma}

\begin{proof}
Suppose that $m,n\in\mathbb{N}^k$, $U$ is an open subset of $\partial\Lambda^{\ge m}$ such that $\sigma_\Lambda^m$ is injective on $U$, and $V$ is an open subset of $\partial\Lambda^{\ge n}$ such that $\sigma_\Lambda^n$ is injective on $V$ and $\sigma_\Lambda^m(U)=\sigma_\Lambda^n(V)$. Define
\begin{align*} 
A=\{(\lambda,\mu,\nu,G):&\lambda\in\Lambda,\ \mu\in\Lambda^m,\ \nu\in\Lambda^n,\ r(\lambda)=s(\mu)=s(\nu),\ \\
& G\text{ is a finite nonexhaustive subset of }s(\lambda)\Lambda, \\
& \ Z_\Lambda(\mu\lambda\ast_s\nu\lambda\setminus G)\subseteq Z(U,m,n,V)\} 
\end{align*}
We claim that $Z(U,m,n,V)=\bigcup_{(\lambda,\mu,\nu,G)\in A}Z_\Lambda(\mu\lambda\ast_s\nu\lambda\setminus G)$, and is thus open in $\G_\Lambda$. It is clear that $\bigcup_{(\lambda,\mu,\nu,G)\in A}Z_\Lambda(\mu\lambda\ast_s\nu\lambda\setminus G)\subseteq Z(U,m,n,V)$. For the other direction, let $(x,m-n,y)\in Z(U,m,n,V)$. Define $\mu :=x(0,m)$ and $\nu=y(0,n)$, and choose $\lambda \in \Lambda$ and a finite nonexhaustive subset $G \subseteq s(\lambda) \Lambda$ such that $\sigma_\Lambda^m(x)=\sigma_\Lambda^n(y)\in Z(\lambda\setminus G)\subseteq \sigma_\Lambda^m(U)=\sigma_\Lambda^n(V)$. Then $(x,m-n,y)\in Z_\Lambda(\mu\lambda\ast_s\nu\lambda\setminus G) \subseteq Z(U,m,n,V)$. Hence $Z(U,m,n,V)\subseteq \bigcup_{(\lambda,\mu,\nu,G)\in A}Z_\Lambda(\mu\lambda\ast_s\nu\lambda\setminus G)$, proving the claim.

Conversely, if $(\lambda,\mu)\in\Lambda\ast_s\Lambda$ and $G$ is a finite nonexhaustive subset of $s(\lambda)\Lambda$, then $Z_\Lambda(\lambda\ast_s\mu\setminus G)=Z(Z_\Lambda(\lambda\setminus G),d(\lambda),d(\mu),Z_\Lambda(\mu\setminus G))$, so the sets $Z(U,m,n,V)$ form a basis for the topology of $\mathcal{G}_\Lambda$.
\end{proof}

The following lemma will be useful for characterising continuous orbit equivalence in the next section. Fix a total order $\preceq$ on $\mathbb{N}^k$ such that $m\le n$ implies that $m\preceq n$. For example, we could let $\preceq$ be the lexicograpical order. Now, let $\Lambda$ be a finitely-aligned $k$-graph, and define $l_\Lambda:\mathcal{G}_\Lambda\to\mathbb{N}^k$ by $$l_\Lambda(x,n,y)=\min\{l\in\mathbb{N}^k: l\ge n\text{ and }\sigma_\Lambda^l(x)=\sigma_\Lambda^{l-n}(y)\},$$ where the minimum is taken with respect to $\preceq$.

\begin{lemma}\label{lem:l}
Let $\Lambda$ be a finitely-aligned $k$-graph. Then $l_\Lambda:\mathcal{G}_\Lambda\to\mathbb{N}^k$ is continuous.
\end{lemma}

\begin{proof}
Suppose that $(x_i,n_i,y_i)_{i\in\N} \to (x,n,y)$ in $\mathcal{G}_\Lambda$. We will first show that $l(x_i,n_i,y_i)\preceq l(x,n,y)$ for large $i$. Choose $p\ge l(x,n,y)$, an open subset $U$ of $\partial\Lambda^{\ge p}$, and an open subset $V$ of $\partial\Lambda^{\ge p-n}$ such that $(x,n,y)\in Z(U,p,p-n,V)$. Then $(x_i,n_i,y_i)\in Z(U,p,p-n,V)$ for large $i$. Let $q=p-l(x,n,y)$. Then
\begin{equation*}
\sigma_\Lambda^q(\sigma_\Lambda^{l(x,n,y)}(x_i))=\sigma_\Lambda^p(x_i)=\sigma_\Lambda^{p-n}(y_i)=\sigma^q_\Lambda(\sigma_\Lambda^{l(x,n,y)-n}(y_i))
\end{equation*}
for large $i$. Since $\sigma_\Lambda^q$ is a local homeomorphism and $$\lim_{i \to \infty} \sigma_\Lambda^{l(x,n,y)}(x_i)=\sigma_\Lambda^{l(x,n,y)}(x)=\sigma_\Lambda^{l(x,n,y)-n}(y)=\lim_{i \to \infty} \sigma_\Lambda^{l(x,n,y)-n}(y_i),$$ it follows that $\sigma_\Lambda^{l(x,n,y)}(x_i)=\sigma_\Lambda^{l(x,n,y)-n}(y_i)$ for large $i$, so $l(x_i,n_i,y_i)\preceq l(x,n,y)$ for large $i$. We now show that $\lim_{i \to \infty} l(x_i,n_i,y_i)= l(x,n,y)$. Since there are only finitely many $l\in\mathbb{N}^k$ for which $l\preceq l(x,n,y)$, it suffices to show that if $l(x_i,n_i,y_i)=l$ for infinitely many $i$, then $l(x,n,y)\preceq l$. So suppose $l(x_i,n_i,y_i)=l$ for infinitely many $i$. Since $\lim_{i \to \infty} n_i=n$, it follows that $l\ge n$. Furthermore, since $\sigma_\Lambda^l(x_i)=\sigma_\Lambda^{l-n_i}(y_i)$ for infinitely many $i$, and $\lim_{i\to\infty} x_i=x$, $\lim_{i\to\infty} n_i=n$, and $\lim_{i\to\infty} y_i=y$, we see that $\sigma_\Lambda^l(x)=\sigma_\Lambda^{l-n}(y)$. Hence $l(x,n,y)\preceq l$ as desired.
\end{proof}

\section{Continuous orbit equivalence}\label{sec:orbit}

In this section, we introduce the notion of continuous orbit equivalence of finitely-aligned $k$-graphs, and characterise continuous orbit equivalence of $k$-graphs in terms of isomorphisms of boundary path groupoids. We then use results of \cite{CR2} and \cite{CRST} to characterise continuous orbit equivalence of $k$-graphs in terms of diagonal-preserving isomorphisms of $k$-graph $C^*$-algebras and also in terms of diagonal-preserving ring-isomorphisms and $*$-algebra-isomorphisms of Kumjian--Pask algebras.

\begin{definition}
Let $\Lambda$ be a finitely-aligned $k$-graph. The \emph{orbit} of $x \in \partial \Lambda$ is given by $$\orb_{\Lambda}(x) := \bigcup_{n \in \N^k} \bigcup_{m \le d(x)} (\sigma_\Lambda^n)^{-1}(\{\sigma_\Lambda^m (x)\}).$$ Let $k_1,k_2 \in \N$, and let $\Lambda_1$ be a finitely-aligned $k_1$-graph and $\Lambda_2$ a finitely-aligned $k_2$-graph. An \emph{orbit equivalence} between $\Lambda_1$ and $\Lambda_2$ is a homeomorphism $h: \partial \Lambda_1 \to \bps_2$ such that $h(\orb_{\Lambda_1}(x_1))=\orb_{\Lambda_2}(h(x_1))$ for all $x_1\in\bps_1$.
\end{definition}

Suppose that $h: \bps_1\to \bps_2$ is an orbit equivalence. For $m_1 \in \N^{k_1}$ and $x_1 \in \bps_1^{\ge m_1}$, we have that $h(\sigma_{\Lambda_1}^{m_1}(x_1)) \in \orb_{\Lambda_2}(h(x_1))$, so there are $f_{m_1}(x_1),\ g_{m_1}(x_1) \in \N^{k_1}$ such that
\begin{align} \label{eqn:orb1}
\sigma_{\Ltwo}^{f_{m_1}(x_1)}(h(\sigma_{\Lone}^{m_1}(x_1)))
=\sigma_{\Ltwo}^{g_{m_1}(x_1)}(h(x_1)).
\end{align}
Similarly, for $m_2 \in \N^{k_2}$ and $x_2 \in \bps_2^{\ge m_2}$, we have that $h^{-1}(\sigma_{\Lambda_2}^{m_2}(x_2)) \in \orb_{\Lambda_1}(h^{-1}(x_2))$, so there are $i_{m_2}(x_2),\ j_{m_2}(x_2) \in \N^{k_2}$ such that
\begin{align} \label{eqn:orb2}
\sigma_{\Lone}^{i_{m_2}(x_2)}(h^{-1}(\sigma_{\Ltwo}^{m_2}(x_2)))
=\sigma_{\Lone}^{j_{m_2}(x_2)}(h^{-1}(x_2)).
\end{align}

\begin{definition}
If $f_{m_1}(x_1),\ g_{m_1}(x_1),\ i_{m_2}(x_2),\ j_{m_2}(x_2)$ can be chosen such that $f_{m_1},\ g_{m_1}: \partial \Lambda_1^{\ge m_1} \to \N^{k_1}$ and $i_{m_2},\ j_{m_2}: \partial \Lambda_2^{\ge m_2} \to \N^{k_2}$ are continuous for all $m_1\in\N^{k_1}$ and all $m_2\in\N^{k_2}$, and such that
\begin{equation}\label{eqn:cocycle1}
g_{m_1+n_1}(x_1)-f_{m_1+n_1}(x_1)=g_{m_1}(x_1)-f_{m_1}(x_1)+g_{n_1}(\sigma_{\Lone}^{m_1}(x_1))-f_{n_1}(\sigma_{\Lone}^{m_1}(x_1))
\end{equation}
for $m_1,\ n_1\in\N^{k_1}$ and $x_1 \in \bps_1^{\ge m_1+n_1}$, and
\begin{equation}\label{eqn:cocycle2}
j_{m_2+n_2}(x_2)-i_{m_2+n_2}(x_2)=j_{m_2}(x_2)-i_{m_2}(x_2)+j_{n_2}(\sigma_{\Ltwo}^{m_2}(x_2))-i_{n_2}(\sigma_{\Ltwo}^{m_2}(x_2))
\end{equation}
for $n_1,n_2\in\N^{k_2}$ and $x_2 \in \bps_2^{\ge n_1+n_2}$, then we say that $h$ is a \emph{continuous orbit equivalence}, that $\Lone$ and $\Ltwo$ are \emph{continuously orbit equivalent}, and that the family $$\{f_{m_1},\ g_{m_1},\ i_{m_2},\ j_{m_2}:m_1\in\N^{k_1},\ m_2\in\N^{k_2}\}$$ is a \emph{family of continuous cocycles for $h$}.
\end{definition}

\begin{exm}\label{exm:coe}
Let $k_1, k_2 \in \N$ and let $\Lambda_1 := \Omega_{k_1,\infty}$ and $\Lambda_2 := \Omega_{k_2,\infty}$. Note that the range maps $r_i:\bps_i\to\N^{k_i}$ are bijections for $i=1,2$. If $\phi: \N^{k_1} \to \N^{k_2}$ is a bijection, then there is a homeomorphism $h: \bps_1 \to \bps_2$ such that $r_2 \circ h = \phi \circ r_1$. We claim that $h$ is a continuous orbit equivalence with family of continuous cocycles given by $f_{m_1}(x_1):=\phi(r_1(x_1))$ and $g_{m_1}(x_1) := \phi(r_1(x_1)+m_1)$ for $m_1 \in \N^{k_1}$ and $x_1 \in \partial \Lambda_1^{\ge m_1}$, and $i_{m_2}(x_2):=\phi^{-1}(r_2(x_2))$ and $j_{m_2}(x_2) := \phi^{-1}(r_2(x_2)+m_2)$ for $m_2 \in \N^{k_2}$ and $x_2 \in \partial \Lambda_2^{\ge m_2}$.

To check the claim, first note that $r_i(\sigma_{\Lambda_i}^{m_i}(x_i))=r_i(x_i)+m_i$ for $i\in\{1,2\}$, $m_i\in\N^{k_i}$ and $x_i\in\partial \Lambda_i^{\ge m_i}$. Using this at the first and fifth equalities, we calculate
\begin{align*}
r_2(\sigma_{\Lambda_2}^{\phi(r_1(x_1)+m_1)}(h(x_1))) &= \phi(r_1(x_1)+m_1)+r_2(h(x_1)) = \phi(r_1(x_1)+m_1)+\phi(r_1(x_1)) \\
&= \phi(r_1(\sigma_{\Lambda_1}^{m_1}(x_1)))+\phi(r_1(x_1)) = r_2(h(\sigma_{\Lambda_2}^{m_1}(x_1)))+\phi(r_1(x_1)) \\
&= r_2(\sigma_{\Lambda_2}^{\phi(r_1(x_1))}(h(\sigma_{\Lambda_1}^{m_1}(x_1))),
\end{align*} 
so \eqref{eqn:orb1} is satisfied since $r_2$ is bijective. A similar calculation shows that \eqref{eqn:orb2} is satisfied. Straightforward calculations show that \eqref{eqn:cocycle1} and \eqref{eqn:cocycle2} are also satisfied. 

It follows from \cite[Example~1.7~(ii)]{KP} that $C^*(\Omega_{k_i,\infty})\cong C^*(\Omega_{k_2,\infty})$ and $\D(\Omega_{k_1,\infty}) \cong \D(\Omega_{k_2,\infty})$. 
\end{exm}

Since we aim to characterise continuous orbit equivalence for arbitrary finitely-aligned higher-rank graphs, we need to consider the periodicity of boundary paths.

\begin{definition}
Let $\Lambda$ be a finitely-aligned $k$-graph and let $x\in\partial\Lambda$. We define the \emph{periodicity group} $\Per(x)$ of $x$ by
$$\Per(x):=\{m-n \in\mathbb{Z}^k: m,n \in \N^k \text{ and } \sigma_{\Lambda}^m(x)=\sigma_{\Lambda}^n(x)\}.$$
Similarly, we define the \emph{inner-peridocity group} $\IP(x)$ of $x$ by
\begin{multline*}
\IP(x):=\{m-n\in \Z^k: m,n \in \mathbb{N}^k \text{ and }\text{there is an open neighbourhood }U\text{ of }x\\\text{such that }\sigma_{\Lambda}^m(y)=\sigma_{\Lambda}^n(y)\text{ for all }y\in U\}.
\end{multline*}
\end{definition}

\begin{remark}\label{iso_per} We denote by $\iso(\G_\Lambda):=\{\eta \in \G_\Lambda: s(\eta)=r(\eta)\}$ the isotropy groupoid of $\G_\Lambda$ and by $\iso(\G_\Lambda)^\circ$ its interior. It is not hard to check that $\iso(\G_\Lambda) = \{(x,n,x): x \in \bps \text{ and } n \in \Per(x)\}$ and $\iso(\G_\Lambda)^\circ=\{(x,n,x): x \in \bps \text{ and } n \in \IP(x)\}$. 
\end{remark}

The following definition describes when a continuous orbit equivalence preserves the periodicity of boundary paths. This definition will become clear in light of Lemma~\ref{lemma:1}. 

\begin{definition} If $h: \bps_1 \to \bps_2$ is a continuous orbit equivalence with a  family of continuous cocycles $\{f_{m_1},\ g_{m_1},\ i_{m_2},\ j_{m_2}:m_1 \in \N^{k_1},\ m_2 \in \N^{k_2} \}$ satisfying
\begin{equation} \label{eqn:per1}
\{g_{m_1}(x_1)-f_{m_1}(x_1)-g_{n_1}(x_1)+f_{n_1}(x_1):m_1-n_1\in\Per(x_1)\}=\Per(h(x_1))
\end{equation}
for every $x_1\in\partial\Lambda_1$, and
\begin{equation} \label{eqn:per2}
\{j_{m_2}(x_2)-i_{m_2}(x_2)-j_{n_2}(x_2)+i_{n_2}(x_2):m_2-n_2\in\Per(x_2)\}=\Per(h^{-1}(x_2))
\end{equation}
for every $x_2\in\partial\Lambda_2$, then we say that $h$ is \emph{period-preserving}.
\end{definition}

We will see that the following type of continuous orbit equivalence is characterised by diagonal-preserving graded isomorphisms of Kumjian--Pask algebras which motivates the terminology.

\begin{definition} 
Let $\Gamma$ be a discrete group and let $\eta_1:\Lambda_1\to\Gamma$ and $\eta_2: \Lambda_2 \to \Gamma$ be functors. We say that a continuous orbit equivalence $h:\bps_1 \to \bps_2$ is $\Gamma$-\emph{graded} if there is a family of continuous cocycles $\{f_{m_1},\ g_{m_1},\ i_{m_2},\ j_{m_2}:m_1\in\N^{k_1},\ m_2\in\N^{k_2}\}$ for $h$ such that
\begin{align} \label{eq:functor1} 
\eta_1(x_1(0,m_1))=\eta_2(h(x_1)(0,g_{m_1}(x_1)))\eta_2(h(\sigma_{\Lone}^{m_1}(x_1))(0,f_{m_1}(x_1)))^{-1}, 
\end{align} 
for every $m_1\in\N^{k_1}$ and every $x_1 \in \bps_1^{\ge m_1}$, and
\begin{align} \label{eq:functor2} \eta_2(x_2(0,m_2))=\eta_1(h^{-1}(x_2)(0,j_{m_2}(x_2)))\eta_1(h^{-1}(\sigma_{\Ltwo}^{m_2}(x_2))(0,i_{m_2}(x_2)))^{-1},
\end{align} 
for every $m_2\in\N^{k_2}$ and every $x_2 \in \bps_2^{\ge m_2}$.
\end{definition}

We now state the main result of this section. If $\Gamma$ is a discrete group and $\eta:\Lambda \to \Gamma$ is a functor, then we denote by $c_\eta$ the continuous cocycle $c_\eta: \G_\Lambda \to \Gamma$ defined by $c_\eta(\mu x, d(\mu)-d(\nu),\nu x) = \eta(\mu) \eta(\nu)^{-1}$.

\begin{proposition} \label{prop:h}
Let $\Gamma$ be a discrete group, $\Lambda_1$ a finitely-aligned $k_1$-graph, $\Lambda_2$ a finitely-aligned $k_2$-graph, $\eta_1:\Lambda_1\to\Gamma$ and $\eta_2:\Lambda_2\to\Gamma$ functors, and $h:\partial\Lambda_1\to\partial\Lambda_2$ a homeomorphism. The following are equivalent.
\begin{enumerate}
	\item There is an isomorphism $\phi:\G_{\Lambda_1}\to\G_{\Lambda_2}$ such that $c_{\eta_2}\circ\phi=c_{\eta_1}$ and $\phi(x_1,0,x_1)=(h(x_1),0,h(x_1))$ for all $x_1\in\partial\Lambda_1$.
	\item The homeomorphism $h$ is a $\Gamma$-graded, period-preserving continuous orbit equivalence.
	\item The homeomorphism $h$ is a $\Gamma$-graded continuous orbit equivalence with a family of continuous cocycles $\{f_{m_1},\ g_{m_1},\ i_{m_2},\ j_{m_2}:m_1\in\N^{k_1},\ m_2\in\N^{k_2}\}$ satisfying
	\begin{equation}\label{eqn:ip1}
		\{g_{m_1}(x_1)-f_{m_1}(x_1)-g_{m_2}(x_1)+f_{m_2}(x_1):m_1-m_2\in\IP(x_1)\}=\IP(h(x_1))
	\end{equation}
	for every $x_1\in\partial\Lambda_1$, and
	\begin{equation}\label{eqn:ip2}
		\{j_{n_1}(x_2)-i_{n_1}(x_2)-j_{n_2}(x_2)+i_{n_2}(x_2):n_1-n_2\in\IP(x_2)\}=\IP(h^{-1}(x_2))
	\end{equation}
	for every $x_2\in\partial\Lambda_2$.
\end{enumerate}
\end{proposition}

To prove Proposition~\ref{prop:h} we need two lemmas but first we give a corollary. A finitely-aligned $k$-graph is called \emph{aperiodic} if for each $v \in \Lambda^0$, there exists a boundary path $x \in v \bps$ such that for all $\lambda,\mu \in \bps v$, $\lambda \not = \mu$ implies that $\lambda x \not = \mu x$. By \cite[Proposition~6.3]{CP}, $\Lambda$ is aperiodic if and only if $\G_\Lambda$ is topologically principal in the sense that the set of units $x \in \G^{0}$ such that $\{\eta \in \iso(\G_\Lambda): s(\eta) =x\} =\{(x,0,x)\}$ is dense in $\G^{0}$.

\begin{corollary}
If $\Lone$ and $\Ltwo$ are aperiodic finitely-aligned higher-rank graphs, then any continuous orbit equivalence $h: \partial \Lone \to \partial \Ltwo$ is period-preserving.
\end{corollary}

\begin{proof}
Let $i\in\{1,2\}$ and $x_i \in \bps_i$. By Remark~\ref{iso_per} we have $\{\eta \in \iso(\G_{\Lambda_i})^\circ: s(\eta)=x_i\}=\{(x_i,n_i,x_i): n_i \in \IP(x_i)\}$. Since $\G_{\Lambda_i}$ is topologically principal it follows that $\IP(x_i)=\{0\}$. So \eqref{eqn:ip1} and \eqref{eqn:ip2} are satisfied trivially. The result now follows from  (3) $\implies$ (2) of Proposition~\ref{prop:h}
\end{proof}

\begin{lemma}\label{lemma:1}
Let $\Lambda_1$ be a finitely-aligned $k_1$-graph and $\Lambda_2$ a finitely-aligned $k_2$-graph, and suppose $h:\partial\Lambda_1\to\partial\Lambda_2$ is a continuous orbit equivalence which has a family of continuous cocycles $\{f_{m_1},\ g_{m_1},\ i_{m_2},\ j_{m_2}:m_1\in\N^{k_1},\ m_2\in\N^{k_2}\}$. There is a continuous groupoid homomorphism $\phi:\G_{\Lambda_1}\to\G_{\Lambda_2}$ such that
\begin{equation}\label{eq:homo}
\phi(x_1,l_1,y_1)=(h(x_1),g_{m_1}(x_1)-f_{m_1}(x_1)-g_{n_1}(y_1)+f_{n_1}(y_1),h(y_1))
\end{equation}
for $m_1,n_1\in\N^{k_1}$, $m_1-n_1=l_1$, $x_1\in\bps_1^{\ge m_1}$, $y_1\in\bps_1^{\ge n_1}$, and $\sigma_{\Lone}^{m_1}(x_1)=\sigma_{\Lone}^{n_1}(y_1)$.
\end{lemma}

\begin{proof} If $m_1,n_1\in\N^{k_1}$, $x_1\in\bps_1^{\ge m_1}$, $y_1\in\bps_1^{\ge n_1}$, $\sigma_{\Lone}^{m_1}(x_1)=\sigma_{\Lone}^{n_1}(y_1)$, then using \eqref{eqn:orb1} at the first and third equalities, we calculate \begin{align*} \sigma_{\Lambda_2}^{f_{n_1}(y_1)+ g_{m_1}(x_1)}(h(x_1)) &= \sigma_{\Lambda_2}^{f_{n_1}(y_1)+f_{m_1}(x_1)}(h(\sigma_{\Lambda_1}^{m_1}(x_1)) \\ & = \sigma_{\Lambda_2}^{f_{m_1}(x_1)+f_{n_1}(y_1)}(h(\sigma_{\Lambda_1}^{n_1}(y_1))) = \sigma_{\Lambda_2}^{f_{m_1}(x_1)+g_{n_1}(y_1)}(h(y_1)). \end{align*} Hence $(h(x_1),g_{m_1}(x_1)-f_{m_1}(x_1)-g_{n_1}(y_1)+f_{n_1}(y_1),h(y_1))\in\G_{\Ltwo}$.

We claim that if $x_1,y_1\in\bps_1$, $m_1,n_1,m'_1,n'_1\in\N^{k_1}$, $\sigma_{\Lone}^{m_1}(x_1)=\sigma_{\Lone}^{n_1}(y_1)$, $\sigma_{\Lone}^{m'_1}(x_1)=\sigma_{\Lone}^{n'_1}(y_1)$, and $m_1-n_1=m'_1-n'_1$, then 
\begin{align*} g_{m_1}(x_1)-f_{m_1}(x_1)-g_{n_1}(y_1)+f_{n_1}(y_1)=g_{m'_1}(x_1)-f_{m'_1}(x_1)-g_{n'_1}(y_1)+f_{n'_1}(y_1).
\end{align*} 
To prove the claim, write $p_{m} := g_{m}-f_{m}$ for $m \in \N^{k_1}$. By \eqref{eqn:cocycle1}, we have that 
\begin{align*} 
p_{m_1+n'_1}(x_1)-p_{n_1+n_1'}(y_1) &= p_{m_1}(x_1)+p_{n_1'}(\sigma_{\Lambda_1}^{m_1}(x_1))-p_{n_1}(y_1)-p_{n_1'}(\sigma_{\Lambda_1}^{n_1}(y_1)) \\ &=p_{m_1}(x_1)-p_{n_1}(y_1).
\end{align*} 
Similarly, $p_{m_1'+n_1}(x_1)-p_{n_1'+n_1}(y_1)=p_{m_1'}(x_1)-p_{n_1'}(y_1)$, so 
\begin{align*} 
p_{m_1}(x_1) - p_{n_1}(y_1) &= p_{m_1+n_1'}(x_1)-p_{n_1+n_1'}(y_1)
\\&= p_{m_1'+n_1}(x_1)-p_{n_1'+n_1}(y_1) = p_{m_1'}(x_1) - p_{n_1'}(y_1),
\end{align*}
proving the claim. Thus, there is a well-defined map $\phi:\G_{\Lambda_1}\to\G_{\Lambda_2}$ such that \eqref{eq:homo} holds for $m_1,n_1\in\N^{k_1}$, $m_1-n_1=l_1$, $x_1\in\bps_1^{\ge m_1}$, $y_1\in\bps_1^{\ge n_1}$, $\sigma_{\Lone}^{m_1}(x_1)=\sigma_{\Lone}^{n_1}(y_1)$. It is straight-forward to check that $\phi$ is a continuous groupoid homomorphism.
\end{proof}

\begin{lemma}\label{lemma:2}
Let $\Lambda_1$ be a finitely-aligned $k_1$-graph and $\Lambda_2$ a finitely-aligned $k_2$-graph, and suppose $h:\partial\Lambda_1\to\partial\Lambda_2$ is a continuous orbit equivalence which has a family of continuous cocycles $\{f_{m_1},\ g_{m_1},\ i_{m_2},\ j_{m_2}:m_1\in\N^{k_1},\ m_2\in\N^{k_2}\}$. For each $x_1\in\bps_1$, there is a group homomorphism $\phi_{\Per(x_1)}:\Per(x_1)\to\Per(h(x_1))$ such that
\begin{equation}\label{eq:group}
\phi_{\Per(x_1)}(m_1-n_1)=g_{m_1}(x_1)-f_{m_1}(x_1)-g_{n_1}(x_1)+f_{n_1}(x_1)
\end{equation}
for $m_1-n_1\in\Per(x_1)$. Moreover, $\phi_{\Per(x_1)}$ restricts to a group homomorphism $\phi_{\IP(x_1)}:\IP(x_1)\to\IP(h(x_1))$.

If \eqref{eqn:per1} holds, then $\phi_{\Per(x_1)}$ is surjective, and if in addition \eqref{eqn:per2} holds for $x_2=h(x_1)$, then $\phi_{\Per(x_1)}$ is an isomorphism.

If \eqref{eqn:ip1} holds, then $\phi_{\IP(x_1)}$ is surjective, and if in addition \eqref{eqn:ip2} holds for $x_2=h(x_1)$, then $\phi_{\IP(x_1)}$ is an isomorphism.
\end{lemma}

\begin{proof}
Fix $x_1\in\bps_1$. Since $\Per(x_1)=\{n_1\in\Z^{k_1}:(x_1,n_1,x_1)\in\G_{\Lone}\}$ and $\Per(h(x_1))=\{n_2\in\Z^{k_2}:(h(x_1),n_2,h(x_1))\in\G_{\Ltwo}\}$, it follows from Lemma~\ref{lemma:1} that there is a group homomorphism $\phi_{\Per(x_1)}:\Per(x_1)\to\Per(h(x_1))$ satisfying \eqref{eq:group} for $m_1-n_1\in\Per(x_1)$.

Suppose $m_1-n_1\in\IP(x_1)$. Then there is an open neighbourhood $U_1$ of $x_1$ such that $\sigma_{\Lone}^{m_1}(y_1)=\sigma_{\Lone}^{n_1}(y_1)$ for all $y_1\in U_1$. Choose an open neighbourhood $U'_1$ of $x_1$ such that $U'_1\subseteq U_1$, $f_{m_1}(y_1)=f_{m_1}(x_1)$, $g_{m_1}(y_1)=g_{m_1}(x_1)$, $f_{n_1}(y_1)=f_{n_1}(x_1)$, and $g_{n_1}(y_1)=g_{n_1}(x_1)$ for every $y_1\in U'_1$. Choose $l_2\in\N^{k_2}$ such that $m_2:=g_{m_1}(x_1)-f_{m_1}(x_1)+l_2\in\N^{k_2}$ and $n_2:=g_{n_1}(x_1)-f_{n_1}(x_1)+l_2\in\N^{k_2}$. Then $h(U'_1)$ is an open neighbourhood of $h(x_1)$ and
\[
\sigma_{\Ltwo}^{m_2}(x_2)
=\sigma_{\Ltwo}^{l_2}(h(\sigma_{\Lone}^{m_1}(h^{-1}(x_2))))
=\sigma_{\Ltwo}^{l_2}(h(\sigma_{\Lone}^{n_1}(h^{-1}(x_2))))
=\sigma_{\Ltwo}^{n_2}(x_2)
\]
for every $x_2\in h(U'_1)$. This shows that $\phi_{\Per(x_1)}(m_1-n_1)=m_2-n_2\in\IP(h(x_1))$. Thus, $\phi_{\Per(x_1)}$ restricts to a group homomorphism $\phi_{\IP(x_1)}:\IP(x_1)\to\IP(h(x_1))$.

It is clear that if \eqref{eqn:per1} holds, then $\phi_{\Per(x_1)}$ is surjective. If in addition \eqref{eqn:per2} holds for $x_2=h(x_1)$, then we have a surjective group homomorphism $\phi'_{\Per(h(x_1))}:\Per(h(x_1))\to\Per(x_1)$. It follows that the rank of $\Per(x_1)$ is equal to the rank of $\Per(h(x_1))$. Since $\Per(h(x_1))$ is a free abelian group, it follows that $\phi_{\Per(x_1)}:\Per(x_1)\to\Per(h(x_1))$ is an isomorphism.

It is clear that if \eqref{eqn:ip1} holds, then $\phi_{\IP(x_1)}$ is surjective. If in addition \eqref{eqn:ip2} holds for $x_2=h(x_1)$, then a similar argument to the one in the previous paragraph shows that $\phi_{\IP(x_1)}$ is an isomorphism.
\end{proof}

For a finitely-aligned $k$-graph $\Lambda$, we denote by $c_\Lambda$ the continuous cocycle $c_\Lambda: \G_\Lambda \to \Z^k$ defined by $c_\Lambda(x,m,y)=m$.

\begin{proof}[Proof of Proposition~\ref{prop:h}.]
$(1)\implies$ (2) and (3): Suppose $\phi:\G_{\Lambda_1}\to\G_{\Lambda_2}$ is an isomorphism such that $c_{\eta_2}\circ\phi=c_{\eta_1}$ and $\phi(x_1,0,x_1)=(h(x_1),0,h(x_1))$ for all $x_1\in\partial\Lambda_1$. Recall the continuous function $l_\Lambda$ of Lemma~\ref{lem:l}. For $m_1\in\N^{k_1}$, we define $g_{m_1}:\partial\Lambda_1^{\ge m_1} \to \N^{k_1}$ by
\[
g_{m_1}(x_1):=l_{\Lambda_2}(\phi(x_1,m_1,\sigma_{\Lone}^{m_1}(x_1)))
\]
and $f_{m_1}:\partial\Lambda_1^{\ge m_1} \to \N^{k_1}$ by
\[
f_{m_1}(x_1):=l_{\Lambda_2}(\phi(x_1,m_1,\sigma_{\Lone}^{m_1}(x_1)))-c_{\Lambda_2}(\phi(x_1,m_1,\sigma_{\Lone}^{m_1}(x_1))).
\]
It follows from the continuity of $l_{\Lambda_1}$, $\phi$ and $c_{\Lambda_2}$ that $f_{m_1}$ and $g_{m_1}$ are continuous. Similarly, for $m_2\in\N^{k_2}$, we define $j_{m_2}:\partial\Lambda_2^{\ge m_2} \to \N^{k_2}$ by
\[
j_{m_2}(x_2):=l_{\Lambda_1}(\phi^{-1}(x_2,m_2,\sigma_{\Ltwo}^{m_2}(x_2)))
\]
and $i_{m_2}:\partial\Lambda_2^{\ge m_2} \to \N^{k_2}$ by
\[
i_{m_2}(x_2):=l_{\Lambda_1}(\phi^{-1}(x_2,m_2,\sigma_{\Ltwo}^{m_2}(x_2)))-c_{\Lambda_1}(\phi^{-1}(x_2,m_2,\sigma_{\Ltwo}^{m_2}(x_2))).
\]
It follows from the continuity of $l_{\Lambda_2}$, $\phi^{-1}$ and $c_{\Lambda_1}$ that $i_{m_2}$ and $j_{m_2}$ are continuous.

We check that $f_{m_1}(x_1)$ and $g_{m_1}(x_1)$ satisfy \eqref{eqn:orb1} for $m_1 \in \N^k$ and $x_1 \in \bps_1^{\ge m_1}$. Write $l:=l_{\Lambda_2}(\phi(x_1,m_1,\sigma_{\Lambda_1}^{m_1}(x_1)))$ and $c := c_{\Lambda_2}(\phi(x_1,m_1,\sigma_{\Lambda_1}^{m_1}(x_1)))$. We have $\phi(x_1,0,x_1) = (h(x_1),0,h(x_1))$, so it follows that $(h(x_1),c,h(\sigma_{\Lambda_1}^{m_1}(x_1)))=\phi(x_1, m_1, \sigma_{\Lambda_1}^{m_1}(x_1))) \in \G_{\Ltwo}$. Thus 
\[
\sigma_{\Lambda_2}^{g_{m_1}(x_1)}(h(x_1))=\sigma_{\Lambda_2}^{l}(h(x_1)) = \sigma_{\Lambda_2}^{l-c}(h(\sigma_{\Lambda_1}^{m_1}(x_1))) = \sigma^{f_{m_1}(x_1)}_{\Lambda_2}(h(\sigma_{\Lambda_1}^{m_1}(x_1))).	
\]

We check that \eqref{eqn:cocycle1} is satisfied for $m_1, n_1 \in \N^k$ and $x_1 \in \partial \Lambda_1^{\ge m_1 + n_1}$. We calculate 
\begin{align*} 
g_{m_1+n_1}(x_1)-f_{m_1+n_1}(x_1) &=
c_{\Lambda_2}(\phi(x_1,m_1+n_1,\sigma_{\Lambda_1}^{m_1+n_1}(x_1))) \\
&= c_{\Lambda_2}(\phi(x_1,m_1,\sigma_{\Lambda_1}^{m_1}(x_1))) +c_{\Lambda_2}(\phi(\sigma_{\Lambda_1}^{m_1}(x_1), n_1, \sigma_{\Lambda_1}^{m_1+n_1}(x_1))) \\
&=g_{m_1}(x_1)-f_{m_1}(x_1) +g_{n_1}(\sigma_{\Lambda_1}^{m_1}(x_1))-f_{n_1}(\sigma_{\Lambda_1}^{m_1}(x_1)). 
\end{align*}

We check that \eqref{eqn:per1} is satisfied for $x_1 \in \bps_1$. If $m_1-n_1 \in \Per(x_1)$, then 
\begin{align*} 
& g_{m_1}(x_1)-f_{m_1}(x_1)-g_{m_1}(x_1)+f_{n_1}(x_1) \\= &c_{\Ltwo}(\phi(x_1,m_1,\sigma_{\Lone}^{m_1}(x_1))) - c_{\Ltwo}(\phi(x_1,n_1,\sigma_{\Lone}^{n_1}(x_1))) \\ =&c_{\Lambda_2}(\phi(x_1,m_1-n_1,x_1)).
\end{align*}
Since $\phi(x_1,0,x_1)=(h(x_1),0,h(x_1))$, it follows that $c_{\Lambda_2}(\phi(x_1,m_1-n_1,x_1)) \in \Per(h(x_1))$. For the other containment, suppose that $m_2 - n_2 \in \Per(h(x_1))$. Since $\phi$ is an isomorphism with $\phi(x_1,0,x_1)=(h(x_1),0,h(x_1))$, it follows that there exists $m_1-n_1 \in \Per(x_1)$ such that $m_2-n_2 = c_{\Lambda_2}(\phi(x_1,m_1-n_1,x_1))$, so \eqref{eqn:per1} holds. A similar argument shows that \eqref{eqn:ip1} also holds.

We check that \eqref{eq:functor1} is satisfied for $x_1 \in \bps_1$ and $m_1 \in \N^k$. Let $l$ and $c$ be as in the second paragraph. Using $c_{\eta_2} \circ \phi = c_{\eta_1}$ and $\sigma_{\Ltwo}^l(h(x_1))=\sigma_{\Ltwo}^{l-c}(h(\sigma_{\Lone}^{m_1}(x_1)))$ at the second equality, we calculate
\begin{align*} 
\eta_1(x_1(0,m_1)) &= c_{\eta_1}(x_1,m_1,\sigma_{\Lambda_1}^{m_1}(x_1)) \\
&= c_{\eta_2}(h(x_1),l-(l-c),h(\sigma_{\Lambda_1}^{m_1}(x_1))) \\
&= \eta_2(h(x_1)(0,l))\eta_2(h(\sigma_{\Lambda_1}^{m_1}(x_1))(0,l-c))^{-1} \\
&=\eta_2(h(x_1)(0,g_{m_1}(x_1)))\eta_2(h(\sigma_{\Lambda_1}^{m_1}(x_1))(0,f_{m_1}(x_1)))^{-1}.
\end{align*}

Similar calculations show that $i_{m_2}$ and $j_{m_2}$ satisfy \eqref{eqn:orb2}, \eqref{eqn:cocycle2}, \eqref{eqn:per2}, \eqref{eq:functor2}, and \eqref{eqn:ip2}.

$(2)\implies (1)$: Suppose $\{f_m,g_m,i_n,j_n:m\in\N^{k_1},\ n\in\N^{k_2}\}$ is a family of continuous cocycles for $h$ such that \eqref{eqn:per1}, \eqref{eqn:per2}, \eqref{eq:functor1}, and \eqref{eq:functor2} hold. By Lemma~\ref{lemma:1} there is a continuous groupoid homomorphism $\phi:\G_{\Lambda_1}\to\G_{\Lambda_2}$ such that
\begin{equation*}
\phi(x_1,l_1,y_1)=(h(x_1),g_{m_1}(x_1)-f_{m_1}(x_1)-g_{n_1}(y_1)+f_{n_1}(y_1),h(y_1)),
\end{equation*}
for $m_1,n_1\in\N^{k_1}$, $x_1\in\bps_1^{\ge m_1}$, $y_1\in\bps_1^{\ge n_1}$, $\sigma_{\Lone}^{m_1}(x_1)=\sigma_{\Lone}^{n_1}(y_1)$, and $m_1-n_1=l_1$. We will show that $\phi$ is bijective, and thus an isomorphism.

To see that $\phi$ is injective, fix $(x_1,l_1,y_1), (x_1',l_1',y_1') \in \G_{\Lone}$ such that $\phi(x_1,l_1,y_1)=\phi (x_1',l_1',y_1')$. Since $h$ is a homeomorphism, it follows that $x_1=x'_1$ and $y_1=y'_1$. It remains to show that $l_1=l_1'$. Observe that 
\begin{align*}
\phi(x_1,l_1-l'_1,x_1)&=\phi(x_1,l_1,y_1)\phi(x'_1,l'_1,y'_1)^{-1}\\ &=\phi(x_1,l_1,y_1)\phi(x_1,l_1,y_1)^{-1}=(h(x_1),0,h(x_1)).
\end{align*}
Let $\phi_{\Per(x_1)}$ be the group homomorphism from $\Per(x_1)$ to $\Per(h(x_1))$ constructed in Lemma~\ref{lemma:2}. By the definitions of $\phi$ and $\phi_{\Per(x_1)}$, we have that $$(h(x_1),\phi_{\Per(x_1)}(l_1-l'_1),h(x_1))=\phi(x_1,l_1-l'_1,x_1)=(h(x_1),0,h(x_1)).$$ Since $\phi_{\Per(x_1)}$ is an isomorphism, we conclude that $l_1-l'_1=0$, and thus that $(x_1,l_1,y_1)=(x'_1,l'_1,y'_1)$, showing that $\phi$ is injective.

To see that $\phi$ is surjective, fix $(x_2,l_2,y_2)\in\G_{\Ltwo}$. Denote by $\phi'$ the continuous groupoid homomorphism obtained by applying Lemma~\ref{lemma:1} to $h^{-1}: \bps_2 \to \bps_1$ and the family $\{i_{m_2}, j_{m_2},\ f_{m_1},\ g_{m_1}:m_2\in\N^{k_2},\ m_1\in\N^{k_1}\}$. Then $\phi(\phi'(x_2,l_2,y_2))=(x_2,l'_2,y_2)$ for some $l'_2\in\Z^{k_2}$. Let $\phi_{\Per(h^{-1}(x_2))}$ be the group homomorphism from $\Per(h^{-1}(x_2))$ to $\Per(x_2)$ constructed in Lemma~\ref{lemma:2}. Since $\phi_{\Per(h^{-1}(x_2))}$ is an isomorphism, it follows that there is an $l_1\in\Per(h^{-1}(x_2))$ such that $\phi_{\Per(h^{-1}(x_2))}(l_1)=l_2-l'_2$. We then have that
\begin{align*}
\phi \big( (h^{-1}(x_2),l_1,h^{-1}(x_2))\phi'(x_2,l_2,y_2) \big) &=
(x_2, \phi_{\Per(h^{-1}(x_2))}(l_1), x_2) \phi(\phi'(x_2, l_2,y_2)) \\&=(x_2,l_2-l'_2,x_2)(x_2,l'_2,y_2)=(x_2,l_2,y_2),
\end{align*}
showing that $\phi$ is surjective.

Since $\phi_{\Per(x)}$ is a group isomorphism, we have that $\phi(x,0,x)=(h(x),0,h(x))$ for $x \in \bps_1$. Finally, we check that $c_{\eta_2} \circ \phi=c_{\eta_1}$. Fix $(x,l,y) \in \G_{\Lone}$ with $l=m-n$. Using the factorisation property at the third equality and \eqref{eq:functor1} at the fourth, we calculate 
\begin{align*}
c_{\eta_2}(\phi (x,l,y)) &= c_{\eta_2}(h(x),g_{m}(x)-f_{m}(x)-g_{n}(y)+f_{n}(y),h(y)) \\
&= \eta_2(h(x)(0,g_m(x)+f_m(x))) \eta_2(h(y)(0,g_n(y)+f_m(y)))^{-1} \\
&= \eta_2(h(x)(0,g_m(x)) h(x)(0,f_n(x)))\eta_2(h(y)(0,g_n(y))h(y)(0,f_m(y)))^{-1} \\
&= \eta_1(x(0,m)) \eta_1(y(0,n))^{-1} = c_{\eta_1}(x,l,y).
\end{align*}

$(3)\implies (1)$: Suppose $\{f_m,g_m,i_n,j_n:m\in\N^{k_1},\ n\in\N^{k_2}\}$ is a family of continuous cocycles for $h$ such that \eqref{eq:functor1}, \eqref{eq:functor2}, \eqref{eqn:ip1}, and \eqref{eqn:ip2} hold. Let $\phi: \G_{\Lambda_1} \to \G_{\Lambda_2}$ be the continuous groupoid homomorphism obtained by Lemma~\ref{lemma:1}. We aim to show that $\phi$ is bijective.

To see that $\phi$ is injective, fix $(x_1,l_1,y_1), (x_1',l_1',y_1') \in \G_{\Lone}$ such that $\phi(x_1,l_1,y_1)=\phi (x_1',l_1',y_1')$. Since $h$ is a homeomorphism, it follows that $x_1=x_1'$ and $y_1=y_1'$. It remains to show that $l_1=l_1'$. We begin by noting that as in the proof of (2) $\implies$ (1), $\phi(x_1,l_1-l_1',x_1)=(h(x_1),0,h(x_1))$. Since $\G_{\Ltwo}$ is étale, its unit space $\G_{\Ltwo}^{0}$ is open in $\G_{\Ltwo}$, and therefore contained in $\iso(\G_{\Ltwo})^\circ$, the interior of the isotropy groupoid. Since $\phi$ is a continuous groupoid homomorphism and $h$ is a homeomorphism, it follows that $\phi^{-1}(\iso(\G_{\Ltwo}))\subseteq \iso(\G_{\Lone})$ and $\phi^{-1}(\iso(\G_{\Ltwo})^\circ)\subseteq \iso(\G_{\Lone})^\circ$. In particular, $(x_1,l_1-l_1',x_1) = \phi^{-1}(h(x_1),0,h(x_1)) \in \iso(\G_{\Lone})^\circ,$ so $l_1-l_1' \in \IP(x_1)$ by Remark~\ref{iso_per}. Let $\phi_{\IP(x_1)}$ be the group homomorphism from $\IP(x_1)$ to $\IP(h(x_1))$ constructed in Lemma~\ref{lemma:2}. By the definitions of $\phi$ and $\phi_{\IP(x_1)}$, we have that $$(h(x_1),\phi_{\IP(x_1)}(l_1-l_1'),h(x_1)) =\phi(x_1,l_1-l_1',x_1)=\phi(h(x_1),0,h(x_1)).$$ Since $\phi_{\IP(x_1)}$ is an isomorphism, we conclude that $l_1=l_1'$, showing that $\phi$ is injective.

To see that $\phi$ is surjective, fix $(x_2, l_2,y_2) \in \G_{\Ltwo}$. Denote by $\phi'$ the continuous groupoid homomorphism obtained by applying Lemma~\ref{lemma:1} to $h^{-1}: \bps_2 \to \bps_1$ and the family $\{i_{m_2}, j_{m_2},\ f_{m_1},\ g_{m_1}:m_2\in\N^{k_2},\ m_1\in\N^{k_1}\}$. Then $\phi(\phi'(x_2,l_2,y_2))=(x_2,l_2',y_2)$ for some $l_2' \in \Z^{k_2}$, and $\phi(\phi'(x_2,l_2,y_2))(y_2,-l_2,x_2)=(x_2,l_2'-l_2,x_2) \in \iso(\G_{\Ltwo}).$ Since $(x_2,l_2,y_2) \in \G_{\Ltwo}$ is fixed, we have that $\phi(\phi'(\zeta))\zeta^{-1} \in \iso(\G_{\Ltwo})$ for every $\zeta \in \G_{\Ltwo}$. Now, we claim that $(x_2,l_2'-l_2,x_2) \in \iso(\G_{\Ltwo})^\circ$. Since $\G_{\Ltwo}$ is \'{e}tale, it follows from the continuity of $\phi$ and $\phi'$ that there is an open bisection $A \subseteq \G_{\Ltwo}$ that contains $(x_2,l_2,y_2)$ such that $\phi(\phi'(\zeta))\zeta^{-1}=(r(\zeta),l_2'-l_2,r(\zeta))$ for every $\zeta\in A$. Since $\{\phi(\phi'(\zeta))\zeta^{-1}:\zeta\in A\}$ is an open subset of $\iso(\G_{\Ltwo})$, it follows that $(x_2,l_2'-l_2,x_2)\in \iso(\G_{\Ltwo})^\circ$. Hence $l_2'-l_2 \in \IP(x_2)$ by Remark~\ref{iso_per}. Let $\phi_{\IP(h^{-1}(x_2))}$ be the group homomorphism from $\IP(h^{-1}(x_2))$ to $\IP(x_2)$ constructed in Lemma~\ref{lemma:2}. Since $\phi_{\IP(h^{-1}(x_2))}$ is an isomorphism, it follows that there is an $l_1\in\IP(h^{-1}(x_2))$ such that $\phi_{\IP(h^{-1}(x_2))}(l_1)=l_2-l'_2$. We then have that
\begin{align*}
\phi \big( (h^{-1}(x_2),l_1,h^{-1}(x_2))\phi'(x_2,l_2,y_2) \big) &=
(x_2, \phi_{\IP(h^{-1}(x_2))}(l_1), x_2) \phi(\phi'(x_2, l_2,y_2)) \\&=(x_2,l_2-l'_2,x_2)(x_2,l'_2,y_2)=(x_2,l_2,y_2),
\end{align*} showing that $\phi$ is surjective.

As in the last paragraph of the proof of (2) $\implies$ (1), $\phi(x,0,x)=(h(x),0,h(x))$ for $x \in \bps_1$ and $c_{\eta_2} \circ \phi = c_{\eta_1}$.
\end{proof}

We now apply results about groupoid $C^*$-algebras \cite{CRST} and Steinberg algebras \cite{CR2} to obtain the following two corollaries. For a discrete group $\Gamma$ and a functor $\eta: \Lambda \to \Gamma$, we denote by $\delta_{\eta}: C^*(\Lambda) \to C^*(\Lambda) \otimes C^*(\Gamma)$ the unique coaction satisfying $\delta_{\eta}(s_\lambda)=s_\lambda \otimes \eta(\lambda)$ for $\lambda \in \Lambda$.

\begin{corollary} \label{conj:1}
Let $\Gamma$ be a discrete group and let $R$ be a reduced indecomposable commutative ring with unit. Let $\Lambda_1$ be a finitely-aligned $k_1$-graph, $\Lambda_2$ a finitely-aligned $k_2$-graph, and let $\eta_1:\Lambda_1\to\Gamma$ and $\eta_2:\Lambda_2\to\Gamma$ be functors. The following are equivalent.
\begin{enumerate}
	\item There is an isomorphism $\phi:\G_{\Lambda_1}\to\G_{\Lambda_2}$ such that $c_{\eta_2}\circ\phi=c_{\eta_1}$.
	\item There is a diagonal-preserving isomorphism $\psi:C^*(\Lambda_1)\to C^*(\Lambda_2)$ such that $\delta_{\eta_2}(\psi(a))=(\psi\otimes\id)(\delta_{\eta_1}(a))$ for $a\in C^*(\Lambda)$.
	\item There is a $\Gamma$-graded diagonal-preserving ring-isomorphism from $\KP_R(\Lambda_1)$ onto $\KP_R(\Lambda_2)$.
	\item There is a $\Gamma$-graded diagonal-preserving $*$-algebra-isomorphism from $\KP_R(\Lambda_1)$ onto $\KP_R(\Lambda_2)$.
	\item There is a $\Gamma$-graded, period-preserving, continuous orbit equivalence $h:\partial\Lambda_1\to\partial\Lambda_2$.
\end{enumerate}
\end{corollary}

\begin{proof}
The equivalence of (1) and (5) follows from Proposition \ref{prop:h}. For $i=1,2$ and $x \in \bps_i$, the isotropy group $(\G_{\Lambda_i})_x^x$ is either trivial or isomorphic to $\Z^{k_i}$, so (1) and (2) are equivalent by \cite[Theorem~6.2]{CRST}, and (1), (3), and (4) are equivalent by \cite[Proposition 2.1 and Theorem 5.7]{S2} and an argument similarly to the one used to prove \cite[Theorem 3.1]{CR2}.
\end{proof}

\begin{corollary} \label{conj:2}
Let $R$ be a reduced indecomposable commutative ring with unit, let $\Lambda_1$ be a finitely-aligned $k_1$-graph, and let $\Lambda_2$ a finitely-aligned $k_2$-graph. The following are equivalent.
\begin{enumerate}
	\item The topological groupoids $\G_{\Lambda_1}$ and $\G_{\Lambda_2}$ are isomorphic.
	\item There is a diagonal-preserving $*$-isomorphism from $C^*(\Lambda_1)$ onto $C^*(\Lambda_2)$.
	\item There is a diagonal-preserving ring-isomorphism from $\KP_R(\Lambda_1)$ onto $\KP_R(\Lambda_2)$.
	\item There is a diagonal-preserving $*$-algebra-isomorphism from $\KP_R(\Lambda_1)$ onto \newline $\KP_R(\Lambda_2)$.
	\item There is a period-preserving continuous orbit equivalence $h: \bps_1 \to \bps_2$.
	\end{enumerate}
\end{corollary}

\begin{proof}
The result follows from Corollary~\ref{conj:1} by letting $\Gamma$ be the trivial group.
\end{proof}

We now show that there are $2$-graphs whose boundary path spaces are homeomorphic but not continuously orbit equivalent.

\begin{lemma}\label{lem:bps}
If $\Lone$ and $\Ltwo$ are $2$-graphs with the same skeleton, then $\bps_1 \cong \bps_2$.
\end{lemma}

\begin{proof}
 Fix a boundary path $x_1: \Omega_{2,m} \to \Lambda_1$ of degree $m \in (\N \cup \{\infty\})^{2}$. Then $x_1$ is uniquely determined by a path in the skeleton of $\Omega_{2,m}$ such that the edges alternate between red and blue edges; this same path in the skeleton also uniquely determines a boundary path $x_2: \Omega_{2,m} \to \Lambda_2$ of degree $m$. A path $\lambda_1 \in \Lambda_1^m$ is identified with a path $\lambda_2 \in \Lambda_2^m$ in the same way. Thus there is a degree-preserving bijective map $\phi: \partial \Lambda_1 \cup \Lambda_1 \to \partial \Lambda_2 \cup \Lambda_2$ such that $\phi(\lambda_1 x)=\phi(\lambda_1)\phi(x)$ for $\lambda_1 \in \Lambda_1$ and $x \in \Lambda_1 \cup \partial \Lambda_1$ with $s(\lambda_1)=r(x)$. We check that $\phi$ is a homeomorphism. If $\lambda_1 \in \Lambda_1$ and $x_1, x_2 \in Z(\lambda_1)$, then $(\phi(x_1))(0,d(\lambda_1))=(\phi(x_2))(0,d(\lambda_1))$, so there exists $\lambda_2 \in \Lambda_2^{d(\lambda_1)}$ such that $\phi(Z_{\Lone}(\lambda_1)) \subseteq Z_{\Ltwo}(\lambda_2)$. If $G_1 \subseteq s(\lambda_1) \Lambda_1$ is a finite subset, then there is a finite subset $G_2 \subseteq s(\lambda_2) \Lambda_2$ such that $\phi(\cup_{\nu \in G_1} Z_{\Lone}( \lambda_1 \nu_1)) \subseteq \cup_{\nu_2 \in G_2} Z_{\Ltwo}(\lambda_2 \nu_2)$, so $\phi(Z_{\Lone}(\lambda_1 \setminus G_1)) \subseteq Z_{\Ltwo}(\lambda_2 \setminus G_2)$. Suppose that $G_1$ is nonexhaustive and choose $\alpha_1 \in s(\lambda_1)\Lone$ such that $\Lambda_1^{\operatorname{min}}(\alpha_1,\beta_1) = \emptyset$ for every $\beta_1 \in G_1$. To show that $G_2$ is nonexhaustive, suppose for a contradiction that there exist $\beta_2 \in G_2$ and $(\rho_2,\tau_2) \in \Lambda_2^{\operatorname{min}}(\phi(\alpha_1), \beta_2)$. Then $(\phi^{-1}(\rho_2),\phi^{-1}(\tau_2)) \in \Lambda_1^{\operatorname{min}}(\alpha_1,\beta_1)$, a contradiction. Finally, the same argument gives that, if $\mu_2 \in \Lambda_2$ and if $G_2 \subseteq s(\lambda_2) \Lambda_2$ is a finite nonexhaustive subset, then there is a $\mu_1 \in \Lambda_1$ and a finite nonexhaustive subset $G_1 \subseteq s(\lambda_1) \Lambda_1$ such that $\phi^{-1}(Z_{\Ltwo}(\mu_2 \setminus G_2)) \subseteq Z_{\Lone}(\mu_1 \setminus G_1)$.
\end{proof}

\begin{remark}\label{rmk:nonexample}
In \cite[Section~6]{KP} it is shown that there are $2$-graphs described by the same skeleton but different factorisation rules that are not isomorphic and whose $C^*$-algebras are not isomorphic. In particular, \cite[Example~6.1]{KP} presents $2$-graphs $\Lone$ and $\Ltwo$ described by the same skeleton such that $C^*(\Lone) \cong \mathcal{O}_2 \not \cong \mathcal{O}_2 \otimes C(\T) \cong C^*(\Ltwo)$. Thus there are $2$-graphs whose boundary path spaces are homeomorphic by Lemma~\ref{lem:bps} but are not continuously orbit equivalent by Corollary~\ref{conj:2}.
\end{remark}

\section{Eventual one-sided conjugacy}\label{sec:eventual}

In this section, we generalise the $1$-graph result \cite[Theorem~4.1]{CR1} by showing that two finitely-aligned $k$-graphs $\Lone$ and $\Ltwo$ are eventually one-sided conjugate if and only if the boundary path groupoids $\G_{\Lone}$ and $\G_{\Ltwo}$ are isomorphic by a groupoid isomorphism that preserves the standard cocycle.

Let $\Lambda$ be a finitely-aligned $k$-graph and let $R$ be a commutative ring with unit. There is a $\Z^k$-grading of $\KP_R(\Lambda)$ given by $\KP_R(\Lambda)_n=\{s_\mu s_\lambda^*:\mu,\nu\in\Lambda,\ d(\mu)-d(\nu)=n\},$ for $n\in\Z^k$, there is a gauge action $\gamma^{\Lambda}$ of $\T^k$ on $C^*(\Lambda)$ given by $\gamma_z^\Lambda(s_\lambda)=z^{d(\lambda)}s_\lambda,$ for $z\in\T^k$ and $\lambda\in\Lambda$, and there is a groupoid cocycle $c_\Lambda:\G_\Lambda\to\Z^k$ given by $c_\Lambda(x,n,y)=n$ for $(x,n,y)\in\G_\Lambda$.

\begin{definition}
Let $\Lone$ and $\Ltwo$ be finitely-aligned $k$-graphs. We say that $\bps_1$ and $\bps_2$ are \emph{eventually one-sided conjugate} if there is a homeomorphism $h:\partial\Lambda_1\to\partial\Lambda_2$ and for each $m \in \N^k$ there are continuous maps $f_m:\partial\Lambda_1^{\ge m}\to\N^k$ and $g_m:\partial\Lambda_2^{\ge m} \to\N^k$ satisfying
$$ \sigma^{f_m(x)}_{\Lambda_2}(h(\sigma_{\Lambda_1}^m(x)))=\sigma^{f_m(x)+m}_{\Lambda_2}(h(x))$$ for $x \in \bps_1^{\ge m}$, and 
$$\sigma^{g_m(y)}_{\Lambda_1}(h^{-1}(\sigma^m_{\Lambda_2}(y)))=\sigma^{g_m(y)+m}_{\Lambda_1}(h^{-1}(y)),$$ for $y\in\bps_2^{\ge m}$. We say that such a homeomorphism $h$ is an \emph{eventual one-sided conjugacy}.
\end{definition}

In particular, if $\bps_1$ and $\bps_2$ are \emph{conjugate}, in the sense that there is a homeomorphism $h: \bps_1 \to \bps_2$ such that $h \circ \sigma_{\Lambda_1}^m = \sigma_{\Lambda_2}^m \circ h$ for all $m \in \N^k$, then they are eventually conjugate.

\begin{theorem}\label{thm:eventual}
Let $\Lambda_1$ and $\Lambda_2$ be finitely-aligned $k$-graphs and let $R$ be a reduced indecomposable commutative ring with unit. The following are equivalent.
\begin{enumerate}
	\item There is an isomorphism $\phi:\G_{\Lambda_1}\to\G_{\Lambda_2}$ such that $c_{\Lambda_2}\circ\phi=c_{\Lambda_1}$.
	\item There is a diagonal-preserving isomorphism $\psi:C^*(\Lambda_1)\to C^*(\Lambda_2)$ such that $\gamma_z^{\Lambda_2}\circ\psi=\psi\circ\gamma_z^{\Lambda_1}$ for all $z\in\T^k$.
	\item There is a $\Z^k$-graded diagonal-preserving isomorphism from $\KP_R(\Lambda_1)$ onto \newline $\KP_R(\Lambda_2)$.
	\item There is a $\Z^k$-graded diagonal-preserving $*$-isomorphism from $\KP_R(\Lambda_1)$ onto $\KP_R(\Lambda_2)$.
	\item There is an eventual one-sided conjugacy $h: \bps_1 \to \bps_2$.
	\end{enumerate}
\end{theorem}

\begin{proof} By letting $k_1=k=k_2$, $\Gamma=\Z^k$, and $\eta_1=d=\eta_2$, we obtain the equivalence of (1)-(4) from Corollary~\ref{conj:1}. We show that (1) and (5) are equivalent. Our proof follows the proof of the 1-graph case in \cite[Theorem~4.1]{CR1}.

(5) $\implies$ (1): Suppose $h: \bps_1 \to \bps_2$ is an eventual one-sided conjugacy. For each $(x,m-n,y) \in \G_{\Lone}$, define $\phi(x,m-n,y) := (h(x), m-n, h(y))$. To see that this gives a well-defined map $\phi: \G_{\Lone} \to \G_{\Ltwo}$, observe that \begin{align*} \sigma_{\Lambda_2}^{f^m(x)+f^n(y)+m}(h(x))&=\sigma_{\Lambda_2}^{f^m(x)+f^n(y)}(h(\sigma^m(x)))\\&=\sigma_{\Lambda_2}^{f^m(x)+f^n(y)}(h(\sigma^n(y)))=\sigma_{\Lambda_2}^{f^m(x)+f^n(y)+n}(h(y)), \end{align*} so $(h(x),m-n,h(y))\in\mathcal{G}_{\Lambda_2}$. It is routine to check that $\phi$ is an isomorphism and it is clear that $c_{\Ltwo} \circ \phi = c_{\Lone}$.

(1) $\implies$ (5): Suppose $\phi:\G_{\Lambda_1} \to\G_{\Lambda_2}$ is an isomorphism such that $c_{\Lambda_2}(\phi(\eta))=c_{\Lambda_1}(\eta)$ for $\eta\in\G_{\Lambda_1}$. Then the restriction of $\phi$ to $\G_{\Lambda_1}^0$ is a homeomorphism onto $\G_{\Lambda_2}^0$. Since the map $x\mapsto (x,0,x)$ is a homeomorphism from $\partial \Lambda_1$ onto $\G_{\Lambda_1}^0$, and $y\mapsto (y,0,y)$ is a homeomorphism from $\partial \Lambda_2$ onto $\G_{\Lambda_2}^0$, it follows that there is a homeomorphism $h:\partial \Lambda_1 \to\partial \Lambda_2$ such that $\phi((x,0,x))=(h(x),0,h(x))$ for all $x\in\partial \Lambda_1$. Since $c_{\Lambda_2}(\phi(\eta))=c_{\Lambda_2}(\eta)$ for all $\eta \in \G_{\Lambda_1}$, it follows that $\phi((x,n,y))=(h(x),n,h(y))$ for all $(x,n,y) \in \G_{\Lambda_1}$. Fix $m \in \N^k$ and $\lambda \in \Lambda_1^m$. Then $A:=Z\big(Z_{\Lone}(\lambda), m, 0, Z_{\Lone}(s(\lambda))\big)$ is a compact open subset of $c_{\Lone}^{-1}(m)$, so $\phi(A)$ is a compact open subset of $c_{\Ltwo}^{-1}(m)$. It follows from Lemma~\ref{lem:groupoidbasis}, that there exist an $n \in \N$, mutually disjoint open subsets $U_1,\dots,U_n$ of $\bps_2$, mutually disjoint open subsets $V_1,\dots,V_n$ of $\bps_2$, and $f_1,\dots,f_n\in\mathbb{N}^k$ such that $$\phi(A)=\bigcup_{i=1}^n Z(U_i,f_i+m,f_i,V_i).$$ Define $f_\lambda:Z(\lambda)\to\mathbb{N}^k$ by $f_\lambda(x)=f_i$ for $x\in h^{-1}(U_i)$. Then $f_\lambda$ is continuous and $\sigma_{\Ltwo}^{f_\lambda(x)}(h(\sigma_{\Lone}^m(x)))=\sigma_{\Ltwo}^{f_\lambda(x)+m}(h(x))$ for $x\in Z(\lambda)$. By repeating this construction for each $\lambda \in \Lambda_1^m$, we get a continuous map $f_m:\bps_1^{\ge m} \to \mathbb{N}^k$ such that $\sigma_{\Ltwo}^{f_m(x)}(h(\sigma_{\Lone}^m(x)))=\sigma_{\Ltwo}^{f_m(x)+m}(h(x))$ for all $x\in \bps_1^{\ge m}$. Applying the same construction gives a continuous map $g_m:\bps_2^{\ge m} \to\mathbb{N}$ such that $\sigma_{\Lone}^{g_m(y)}(h^{-1}(\sigma_{\Ltwo}^m(y)))=\sigma_{\Lone}^{g_m(y)+m}(h^{-1}(y))$ for all $y\in \bps_2^{\ge m}$. Hence $h$ is an eventual one-sided conjugacy. \end{proof}

\section{Stable isomorphism}\label{sec:stable}

In this section, we characterise stable isomorphism of the $C^*$-algebras and Kumjian--Pask algebras of finitely-aligned $k$-graphs in terms of isomorphisms of boundary path groupoids by applying results about groupoid $C^*$-algebras and Steinberg algebras \cite{CR2, CRS,CRST}. We first study the stabilisation of $k$-graphs (see \cite{T2} for the $1$-graph case).

\begin{definition} Let $\Lambda$ be a finitely-aligned $k$-graph. The \emph{stabilisation of} $\Lambda$, denoted by $S\Lambda$, is the $k$-graph obtained by attaching a copy of $\Omega_{k,\infty}$ to each vertex. We identify each $v\in\Lambda^0$ with the vertex $(0,0)$ in the copy of $\Omega_{k,\infty}$ corresponding to $v$.
\end{definition}

Denote by $\R$ the full countable equivalence relation $\R = \N \times \N$, regarded as a discrete principal groupoid with unit space $\N$. Similarly, for $k \in \N$, denote by $\R_k$ the full countable equivalence relation $\R_k = \N^k \times \N^k$, regarded as a discrete principal groupoid with unit space $\N^k$.

\begin{lemma}\label{lem:stablegroupoid} 
Let $\Lambda$ be a finitely-aligned $k$-graph. Then $\G_{S \Lambda} \cong \G_\Lambda \times \R$. 
\end{lemma}

\begin{proof} 
We first show that $\G_{S\Lambda} \cong \G_\Lambda\times\R_k$ and then show that $\R_k \cong \R$. Note that for each $n\in\mathbb{N}^k$, there is a unique $\mu_n\in \Omega_{k,\infty}$ such that $r(\mu_n)=(n,n)$ and $s(\mu_n)=(0,0)$. For each $v \in \Lambda^0$, write $\mu_{n,v}$ for $\mu_n$ in the copy of $\Omega_{k,\infty}$ attached to $v$. Then $\partial (S\Lambda) = \{ \mu_{n,r(x)} x: x \in \bps, n \in \N^k \}$ and has a basis consisting of compact open sets of the form $Z_{S\Lambda}(\mu_{n,r(\lambda)} \lambda \setminus G)$. The map $\phi: \bps \times \N^k \to \partial (S \Lambda)$ defined by $\phi((x,n))=\mu_{n,r(x)} x$ restricts to a continuous bijection of $Z_{S\Lambda}(\mu_{n,r(\lambda)} \lambda \setminus G)$ onto the compact open set $Z_\Lambda(\lambda \setminus G) \times \{ n \}$, so $\phi$ is a homeomorphism. It is routine to check that $$((x,m,y),(p,q))\mapsto (\mu_{p,r(x)} x,m+p-q,\mu_{q,r(y)} y)$$ defines a groupoid isomorphism from $\G_\Lambda \times \R_k$ to $\G_{S\Lambda}$. To see that $\mathcal{R}_k \cong \mathcal{R}$, let $\phi$ be any bijection from $\mathbb{N}^k$ to $\mathbb{N}$. Then $(m,n)\mapsto (\phi(m),\phi(n))$ is an isomorphism from $\mathcal{R}_k$ to $\mathcal{R}$. 
\end{proof}

We denote by $\K$ the compact operators on $\ell^2(\N)$, and by $\mathcal{C}$ the maximal abelian subalgebra of $\K$ consisting of diagonal operators. For a commutative ring $R$ with unit, we denote by $M_\infty(R)$ the ring of finitely supported, countable infinite square matrices over $R$, and by $D_\infty(R)$ the abelian subring of $M_\infty(R)$ consisting of diagonal matrices.

\begin{remark}
Let $\Lambda$ be a finitely-aligned $k$-graph and $R$ a commutative ring with unit. It follows from Lemma~\ref{lem:stablegroupoid} that $C^*(S \Lambda) \cong C^*(\Lambda) \otimes \K$ and $\KP_R(S \Lambda) \cong \KP_R(\Lambda) \otimes M_\infty(R)$. Hence the class of $k$-graph $C^*$-algebras and the class of Kumjian--Pask algebras are closed under stabilisation.
\end{remark}

We now apply results of \cite{CR2, CRS,CRST}. We first recall some definitions. Let $\G$ be an ample groupoid with unit space $\G^{0}$. If $X$ is a subset of $\G^{0}$, then we let $\G|_X:=s^{-1}(X)\cap r^{-1}(X)$, and say that $X$ is \emph{$\G$-full} (or just \emph{full}) if $r(s^{-1}(X))=\G^{0}$. Two ample groupoids $\G_1$ and $\G_2$ are \emph{Kakutani equivalent} if, for $i=1,2$, there is a $\G_i$-full clopen subset $X_i\subseteq\G_i^{0}$ such that $\G_1|_{X_1}$ and $\G_2|_{X_2}$ are isomorphic (see \cite{CRS} and \cite{Matui}). We say that $\G_1$ and $\G_2$ are \emph{groupoid equivalent} if there is a $\G_1$--$\G_2$ equivalence in the sense of \cite[Definition 2.1]{MRW}.

For a finitely-aligned $k_1$-graph and a finitely-aligned $k_2$-graph, we say that an isomorphism $\phi:C^*(\Lone)\otimes\mathcal{K}\to C^*(\Ltwo)\otimes\mathcal{K}$ is \emph{diagonal-preserving} if $\phi(\mathcal{D}(\Lone)\otimes\mathcal{C})=\mathcal{D}(\Ltwo)\otimes\mathcal{C}$. Likewise, we say that an isomorphism $\phi:\KP_R(\Lone)\otimes M_\infty(R)\to \KP_R(\Ltwo)\otimes M_\infty(R)$ is \emph{diagonal-preserving} if $\phi(D_R(\Lone)\otimes D_\infty(R))=D_R(\Ltwo)\otimes D_\infty(R)$.

For a ring $A$, we denote by $M(A)$ the \emph{multiplier ring} of $A$ (see for example \cite{AP}).

\begin{corollary}\label{cor:kakutani}
Let $R$ be a reduced indecomposable commutative ring with unit, let $\Lambda_1$ be a finitely-aligned $k_1$-graph, and let $\Lambda_2$ a finitely-aligned $k_2$-graph. The following are equivalent.
\begin{enumerate}
	\item The groupoids $\G_{S\Lambda_1}$ and $\G_{S\Lambda_2}$ are isomorphic.				
	\item The groupoids $\G_{\Lambda_1}\times\R$ and $\G_{\Lambda_2}\times\R$ are isomorphic.
	\item The groupoids $\G_{\Lambda_1}$ and $\G_{\Lambda_2}$ are Kakutani equivalent.
	\item The groupoids $\G_{\Lambda_1}$ and $\G_{\Lambda_2}$ are groupoid equivalent.
	\item There is a diagonal-preserving isomorphism from $C^*(S\Lambda_1)$ onto $C^*(S\Lambda_2)$.
	\item There is a diagonal-preserving ring-isomorphism from $\KP_R(S\Lambda_1)$ onto $\KP_R(S\Lambda_2)$.
	\item There is a diagonal-preserving $*$-algebra-isomorphism from $\KP_R(S\Lambda_1)$ onto $\KP_R(S\Lambda_2)$.
	\item There is a period-preserving continuous orbit equivalence $h:\partial (S\Lambda_1)\to \partial (S\Lambda_2)$.
	\item There is a diagonal-preserving $*$-isomorphism from $C^*(\Lambda_1)\otimes\K$ onto $C^*(\Lambda_2)\otimes K$.
	\item For $i=1,2$, there is a projection $p_i\in M(D(\Lambda_i))$ such that $p_i$ is full in $C^*(\Lambda_i)$ and an isomorphism from $\psi:p_1C^*(\Lambda_1)p_1\to p_2C^*(\Lambda_2)p_2$ such that $\psi(p_1 D(\Lambda_1))=p_2 D(\Lambda_2)$. 				\item There is a diagonal-preserving ring-isomorphism from $\KP_R(\Lambda_1)\otimes M_\infty(R)$ onto \newline $\KP_R(\Lambda_2)\otimes M_\infty(R)$.
	\item For $i = 1,2$, there is an idempotent $p_i\in M(D_R(\Lambda_i))$ such that $p_i$ is full in $\KP_R(\Lambda_i)$, and a ring-isomorphism from $\psi:p_1\KP_R(\Lambda_1)p_1\to p_2\KP_R(\Lambda_2)p_2$ such that $\psi(p_1 D_R(\Lambda_1))=p_2 D_R(\Lambda_2)$.
	\item There is a diagonal-preserving $*$-algebra-isomorphism from $\KP_R(\Lambda_1)\otimes M_\infty(R)$ onto $\KP_R(\Lambda_2)\otimes M_\infty(R)$.
	\item For $i=1,2$, there is a projection $p_i\in M(D_R(\Lambda_i))$ such that $p_i$ is full in $\KP_R(\Lambda_i)$ and a $*$-algebra-isomorphism $ \psi:p_1\KP_R(\Lambda_1)p_1\to p_2\KP_R(\Lambda_2)p_2$ such that \newline $\psi(p_1 D_R(\Lambda_1))=p_2 D_R(\Lambda_2)$.
\end{enumerate}
\end{corollary}

\begin{proof} 
The equivalence of (1) and (2) follows from Lemma~\ref{lem:stablegroupoid}. The equivalence of (2)-(4) is an application of \cite[Theorem~3.2]{CRS}. The equivalence of (1) and (5)-(8) follows from Corollary~\ref{conj:2}. For $i=1,2$ and $x\in \bps_i$, we have that $(\G_{\Lambda_i})_x^x$ is either trivial or isomorphic to $\Z^{k_i}$, so the equivalence of (2)-(4) and (11)-(14) follows from \cite[Proposition 2.1 and Theorem 5.7]{S2} and an argument similarly to the one used to prove \cite[Corollary~3.12]{CR2}, and the equivalence of (2)-(4) and (9)-(10) is an application of \cite[Corollary~11.3]{CRST}.
\end{proof}

\section{Two-sided conjugacy}\label{sec:twosided}

In this section, we generalise the $1$-graph result \cite[Theorem~5.1]{CR1} by showing that the two-sided multi-dimensional shift spaces associated to row-finite $k$-graphs $\Lone$ and $\Ltwo$ with finitely many vertices and no sinks or sources are conjugate if and only if the boundary path groupoids $\G_{S\Lone}$ and $\G_{S\Ltwo}$ are isomorphic by a cocycle-preserving isomorphism.

Let $\Lambda$ be a finitely-aligned $k$-graph, let $\Gamma$ be a discrete group, and let $\eta:\Lambda\to\Gamma$ be a functor. If $\lambda\in S\Lambda$, then either $\lambda\in\Lambda$, or there is a unique $\lambda_1\in\Lambda$ and a unique $\lambda_2$ in the copy of $\Omega_{k,\infty}$ corresponding to the vertex $r(\lambda_1)$ such that $\lambda=\lambda_2\lambda_1$. We write $\eta_\Lambda$ for the functor from $S\Lambda$ to $\Z^k$ given by $\eta_\Lambda(\lambda)=d(\lambda)$ for $\lambda\in\Lambda$, and $\eta_\Lambda(\lambda)=d(\lambda_1)$ for $\lambda\in S\Lambda\setminus\Lambda$.

Let $\overline{\Omega}_k$ be the $k$-graph $\overline{\Omega}_k := \{(m,n)\in\Z^k\times\Z^k:m\le n\}$ with degree map $d:\overline{\Omega}_k\to\N^k$ defined by $d(m,n)=n-m$, and range and source maps $r,s$ defined by $r(m,n)=(m,m)$ and $s(m,n)=(n,n)$. If $\Lambda$ is row-finite and has finitely many vertices and no sinks or sources, then we write $\X_\Lambda$ for the space of all $k$-graph morphisms from $\overline{\Omega}_k$ to $\Lambda$. We equip $\X_\Lambda$ with the topology generated by subsets of the form $$Z((m,n),\lambda) := \{x\in \X_\Lambda:x(m,n)=\lambda\}$$ where $(m,n)\in \overline{\Omega}_k$ and $\lambda\in\Lambda^{n-m}$. For $m\in\Z^k$, we denote by $\tsh_\Lambda^m$ the homeomorphism from $\X_\Lambda$ to itself given by $$\tsh_\Lambda^m(x)(p,q)=x(p+m,q+m)$$ for $x\in\X_\Lambda$ and $(p,q)\in\overline{\Omega}_k$.

If $\G$ is a groupoid, $\Gamma$ is a discrete group, and $c:\G\to\Gamma$ is a cocycle, then we write $\bar{c}$ for the cocycle from $\G\times\R$ to $\Gamma$ given by $\bar{c}(\zeta_1,\zeta_2)=c(\zeta_1)$.

\begin{theorem}\label{thm:conjugacy}
Let $\Lambda_1$ and $\Lambda_2$ be finitely-aligned $k$-graphs and let $R$ be a reduced indecomposable commutative ring with unit. The following are equivalent.
\begin{enumerate}
	\item There is an isomorphism $\phi:\G_{S\Lambda_1}\to\G_{S\Lambda_2}$ such that $c_{\eta_{\Lambda_2}}\circ\phi=c_{\eta_{\Lambda_1}}$.
	\item There is an isomorphism $\phi:\G_{\Lambda_1}\times\R\to\G_{\Lambda_2}\times\R$ such that $\bar{c}_{\Lambda_2}\circ\phi=\bar{c}_{\Lambda_1}$.
	\item There is a diagonal-preserving $*$-isomorphism $\psi:C^*(\Lambda_1)\otimes\K\to C^*(\Lambda_2)\otimes\K$ such that $(\gamma_z^{\Lambda_2}\otimes\id)\circ\psi=\psi\circ(\gamma_z^{\Lambda_1}\otimes\id)$ for all $z\in\T^k$.
	\item There is a $\Z^k$-graded diagonal-preserving ring-isomorphism $\psi:\KP_R(\Lambda_1)\otimes M_\infty(R)\to\KP_R(\Lambda_2)\otimes M_\infty(R)$. 
	\item There is a $\Z^k$-graded diagonal-preserving $*$-algebra-isomorphism $\psi:\KP_R(\Lambda_1)\otimes M_\infty(R) \to\KP_R(\Lambda_2)\otimes M_\infty(R)$. 
\end{enumerate}
Moreover, if $\Lambda_1$ and $\Lambda_2$ are row-finite and have finitely many vertices and no sinks or sources, then (1)-(5) are equivalent to the following.
\begin{enumerate}
	\item[(6)] There is a homeomorphism $h:\X_{\Lambda_1}\to\X_{\Lambda_2}$ such that $\tsh^m_{\Lambda_2}(h(x))=h(\tsh_{\Lambda_1}^m(x))$ for all $m\in\N^k$ and all $x\in\X_{\Lambda_1}$.
\end{enumerate}
\end{theorem}

\begin{proof}
For $i=1,2$, let $\kappa_i:\G_{S\Lambda_i}\to\G_{\Lambda_i}\times\R$ be the isomorphism of Lemma~\ref{lem:stablegroupoid}. We have that $\bar{c}_{\Lambda_i}((x,m,y),(p,q))=m=c_{\eta_{\Lambda_i}}(\mu_{p,r(x)}x,m+p-q,\mu_{q,r(y)}y)$, so $\bar{c}_{\Lambda_i}\circ\kappa_i=c_{\eta_{\Lambda_i}}$. It follows that (1) and (2) are equivalent. The equivalence of (2), (4) and (5) follows from \cite[Proposition 2.1 and Theorem 5.7]{S2} and an argument similarly to the one used to prove \cite[Theorem 3.11]{CR2}. The equivalence of (2) and (3) follows from \cite[Corollary~11.3]{CRST}. It remains to show that (1) and (6) are equivalent when $\Lambda_1$ and $\Lambda_2$ are row-finite $k$-graphs that have finitely many vertices and no sinks or sources. Our proof follows the proof of the 1-graph case in \cite[Theorem~5.1]{CR1}. We identify $\partial (S \Lambda)$ with $\partial \Lambda \times \N^k$ via the isomorphism used in the proof of Lemma~\ref{lem:stablegroupoid} given by $\mu_{r(x),n} x \to (x,n)$, where $x \in \partial \Lambda$ and $n \in \N^k$.
	
(1) $\implies$ (6): Suppose $\Phi: \G_{S \Lone} \to \G_{S\Ltwo}$ is an isomorphism satisfying $c_{\eta_{\Ltwo}}(\Phi(\gamma)) = c_{\eta_{\Lone}}(\gamma)$ for $\gamma \in \G_{S\Lone}$. For $x \in \Lambda_1^\infty$, we have $(x,0) \in (S\Lone)^\infty$ and $((x,0), 0, (x,0)) \in \G_{S\Lone}$. Since $\Phi$ is cocycle-preserving, we have $\Phi((x,0),0,(x,0)) = ((y, p),0,(y, p))$ for some uniquely determined $y \in \Lambda_2^\infty$ and $p \in \N^k$. Define $\psi: \Lambda_1^\infty \to \Lambda_2^\infty$ by $\psi(x) := y$. Since $\Phi$ is continuous, the map $(x,0) \mapsto (y, p)$ is continuous, so $\psi$ is also continuous.

Fix $m \in \N^k$. We have $((x,0), m, (\sigma_{\Lone}^m(x),0)) \in \G_{S\Lone}$ for $x \in \Lambda_1^\infty$. Since $\Phi$ is cocycle-preserving, there exist $p,q \in \N^k$ such that $$\Phi((x,0), m, (\sigma_{\Lone}^m(x),0)) = ((\psi(x), p), m+p-q, (\psi(\sigma_{\Lone}^m(x)), q)) \in \G_{S\Lone}.$$ Hence there exists $l \in \N^k$ such that $\sigma_{\Ltwo}^{l+m} (\psi(x)) = \sigma_{\Ltwo}^l(\psi(\sigma_{\Lone}^m(x)))$; let $l(x) \in \N^k$ be such that $l(x) \le l$ for all $l \in \N^k$ satisfying this identity. We check that $l: \Lambda_1^\infty \to \N^k$ is continuous. Suppose $(x^i)_{i\in \N}$ in $\Lambda_1^\infty$ converges to $x$. Then $\Phi((x^i,0), m, (\sigma_{\Lone}^m(x^i),0))\to \Phi((x,0), m, (\sigma_{\Lone}^m(x),0))$ since $\Phi$ is continuous and $\psi(x^i)\to \psi(x)$ and $\psi(\sigma_{\Lone}^m(x^i))\to \psi(\sigma_E(x))$ since $\psi$ is continuous. It follows that there is an $N\in\N$ such that for $i\ge N$, we have that $\sigma_{\Ltwo}^{l(x)+m} (\psi(x^i)) = \sigma_{\Ltwo}^{l(x)}(\psi(\sigma_{\Lone}(x^i)))$ and either $l(x)=0$ or $\sigma_{\Ltwo}^{l(x)} (\psi(x^i)) \ne \sigma_{\Ltwo}^{l(x)-1}(\psi(\sigma_{\Lone}(x^i)))$. Hence $l(x^i)=l(x)$ for $i\ge N$.

Since $\Lambda_1^\infty$ is compact, it follows that there is an $L \in \N^k$ such that $\sigma_{\Ltwo}^{L+m}(\psi(x)) = \sigma_{\Ltwo}^L(\psi(\sigma_{\Lone}^m(x)))$ for all $x \in \Lambda_1^\infty$. Define $\varphi:= \sigma_{\Ltwo}^L \circ \psi: \Lambda_1^\infty \to \Ltwo^\infty$. Then $\varphi$ is continuous and satisfies $\varphi \circ \sigma_{\Lone}^m = \sigma_{\Ltwo}^m \circ \varphi$.

Now, define $\overline \varphi: \X_{\Lone} \to \X_{\Ltwo}$ by $(\overline \varphi(x))(p,\infty)=\varphi(x(p,\infty))$ for $x \in \X_{\Lone}$ and $p \in \Z^k$. Since $\varphi \circ \sigma_{\Lone}^m= \sigma_{\Ltwo}^m \circ \varphi$ for all $m \in \N^k$, it follows that $\overline \varphi$ is well-defined. It is routine to check that $\overline \varphi$ is continuous and that $\overline \varphi \circ \overline \sigma_{\Lone}^m = \overline \sigma_{\Ltwo}^m \circ \overline \varphi$ for all $m \in \N^k$. We will show that $\overline\varphi$ is bijective. It will then follow that $\overline\varphi$ is a conjugacy and thus that $\X_{\Lone}$ and $\X_{\Ltwo}$ are conjugate.

We first show that $\overline\varphi$ is injective. Take $x,x' \in \Lambda_1^\infty$ such that $\varphi(x)=\varphi(x')$, and choose $p,p'\in\mathbb{N}^k$ such that $\Phi((x,0),0,(x,0))=((\psi(x),p),0,(\psi(x),p))$ and $\Phi((x',0),0,(x',0))=((\psi(x'),p'),0,(\psi(x'),p'))$. Since
$\sigma_{\Ltwo}^L(\psi(x))=\sigma_{\Ltwo}^L(\psi(x')),$
it follows that $((\psi(x),p),p-p',(\psi(x'),p'))\in\G_{S\Ltwo}$ and thus that
$$((x,0),0,(x',0))=\Phi^{-1}(((\psi(x),p),p-p',(\psi(x'),p')))\in\G_{S\Lone}.$$
It follows that there is an $n\in\mathbb{N}^k$ such that $\sigma_{\Lone}^n(x)=\sigma_{\Lone}^n(x')$. Let $n((x,x'))$ be such that $n((x,x')) \le n$ for all $n \in \N^k$ such that $\sigma_{\Lone}^n(x)=\sigma_{\Lone}^n(x')$. An argument similar to the one used to prove that $l:\Lambda_1^\infty\to\N$ is continuous, shows that $$n:\{(x,x')\in \Lambda_1^\infty\times \Lambda_1^\infty:\varphi(x)=\varphi(x')\}\to\mathbb{N}^k$$ is continuous. Since $\{(x,x')\in \Lambda_1^\infty\times \Lambda_1^\infty:\varphi(x)=\varphi(x')\}$ is closed in $\Lambda_1^\infty\times \Lambda_1^\infty$ and thus compact, there exists $N\in\mathbb{N}$ such that $\sigma_{\Lone}^N(x)=\sigma_{\Lone}^N(x')$ for all $x,x'\in \Lambda_1^\infty$ satisfying $\varphi(x)=\varphi(x')$. The injectivity of $\overline\varphi$ follows.

Next, we show that $\overline\varphi$ is surjective. Suppose $y\in \Lambda_2^\infty$. Then $$\Phi^{-1}((y,0),0,(y,0))=((x,n),0,(x,n))$$ for some $x\in \Lambda_1^\infty$ and some $n\in\mathbb{N}^k$. Choose $p\in\mathbb{N}^k$ such that $$\Phi((x,0),0,(x,0))=((\psi(x),p),0,(\psi(x),p)).$$ Since $((x,0),-n,(x,n))\in\G_{S\Lone}$ and $\Phi((x,0),-n,(x,n))=((\psi(x),m),m,(y,0))$, it follows that there exists $j\in\mathbb{N}^k$ such that $\sigma_{\Ltwo}^j(\psi(x))=\sigma_{\Ltwo}^j(y)$. An argument similar to the one used in the previous paragraph, then shows that there exists $J\in\mathbb{N}^k$ such that for each $y\in \Lambda_2^\infty$ there is an $x\in \Lambda_1^\infty$ such that $\sigma_{\Lambda_2}^J(\psi(x))=\sigma_{\Lambda_2}^J(y)$. The surjectivity of $\overline\varphi$ follows.

(6) $\implies$ (1): Suppose there is a conjugacy $h: \X_{\Lone} \to \X_{\Ltwo}$. Then there is an $l \in \N$ such that if $x, x' \in \X_{\Lone}$ satisfy $x(0,n) = x'(0,n)$ for all $n \in \N^k$, then $(h(x))(0,n) = (h(x'))(0,n)$ for all $n \in \N^k$, and if $y,y' \in \X_{\Ltwo}$ satisfy $y(0,n) = y'(0,n)$ for all $n \in \N^k$, then $(h^{-1}(y))(0,n) = (h^{-1}(y'))(0,n)$ for all $n \ge l$. It follows that there is a continuous map $\pi: \Lambda_1^\infty \to \Lambda_2^\infty$ such that $(\pi(x(0,\infty)))(0,n)=(h(x))(0,n)$ for $x \in \X_{\Lone}$ and $n\in\mathbb{N}^k$. Then $\pi$ is surjective, $\pi \circ \sigma_{\Lone}^m = \sigma_{\Ltwo}^m \circ \pi$ for all $m \in \N^k$, and if $\pi(x) = \pi(x')$ for $x, x' \in \Lambda_1^\infty$, then $\sigma_{\Lone}^l(x)=\sigma_{\Lone}^l(x')$.

It follows from the continuity of $\pi$ and the compactness of $\Lambda_1^\infty$ that we can choose $L\ge l$ such that if $x(0,L)=x'(0,L)$, then $(\pi(x))(0,l)=(\pi(x'))(0,l)$. Define an equivalence relation $\sim$ on $\Lambda^L$ by $\lambda\sim\mu$ if there are $x\in Z(\lambda)$ and $x'\in Z(\mu)$ such that $\pi(x) = \pi(x')$ (that $\sim$ is transitive follows from the fact that if $\lambda,\mu,\eta\in \Lambda_1^L$, $x,x'\in \Lambda_1^\infty$, and $\pi(\lambda x)=\pi(\mu x)$ and $\pi(\mu x')=\pi(\eta x')$, then $\pi(\lambda x')=\pi(\eta x')$). Then $\pi(x) = \pi(x')$ if and only if $x(0,L)\sim x'(0,L)$ and $\sigma_{\Lone}^l(x)=\sigma_{\Lone}^l(x')$.

For each equivalence class $B\in \Lambda_1^L/\sim$ choose a partition $\{A_\lambda: \lambda \in B \}$ of $\N^k$ and bijections $f_\lambda: A_\lambda \to \N^k$. The map $\psi: (x,n) \mapsto (\pi (x), f^{-1}_{x(0,L)} (n))$ is then a homeomorphism from $(S\Lone)^\infty \to (S\Ltwo)^\infty$. It is routine to check that
$$\Phi:\left((x,n),p,(x',n')\right) \mapsto \left(\psi(x,n),p+n'+f^{-1}_{x(0,L)} (n)-n-f^{-1}_{x'(0,L)} (n'),\psi(x',n')\right)$$ is a groupoid isomorphism from $\G_{S\Lone}$ to $\G_{S\Ltwo}$ satisfying $c_{\eta_{\Ltwo}}(\Phi(\gamma)) = c_{\eta_{\Lone}}(\gamma)$ for $\gamma \in \G_{S\Lone}$. \end{proof}

\section*{Acknowledgements}

JR is grateful for financial support from R. Hazrat (WSU), D. Pask (UoW) and A. Sims (UoW) via their jointly held ARC grant DP150101598, O. Shalit (Technion), the Institute of Mathematics and its Applications at the University of Wollongong, and his father Don. JR is grateful for the hospitality of the Carlsen family and G. Restorff while visiting and attending a workshop at the Faroe Islands.

\end{document}